\documentclass[11pt]{amsart}
\usepackage[titletoc,toc,title]{appendix}
\usepackage[left=2.7cm,right=2.7cm,top=3.5cm,bottom=3cm]{geometry}
\usepackage{amssymb,latexsym,amsmath,amsthm,amscd}
\usepackage{graphicx}
\usepackage{dsfont}
\usepackage[all]{xy}
\usepackage{mathrsfs}
\usepackage{color}
\definecolor{Blue}{rgb}{0.3,0.3,0.9}

\usepackage[T2A,OT1]{fontenc}

\newcommand{\sk}{\vspace{0.1in}}
\newtheorem{thm}{Theorem}[section]
\newtheorem{def-thm}[thm]{Definition-Theorem}
\newtheorem{cor}[thm]{Corollary}
\newtheorem{lem}[thm]{Lemma}

\newtheorem{def-lem}[thm]{Definition-Lemma}
\newtheorem{prop}[thm]{Proposition}

\newtheorem*{ThmA}{Theorem A}
\newtheorem*{ThmB}{Theorem B}

\newtheorem*{intro-conj1}{Conjecture}
\newtheorem*{intro-conj2}{Conjecture}
\newtheorem*{intro-conj3}{Conjecture}
\newtheorem*{intro-conj4}{Conjecture}
\theoremstyle{definition}
\newtheorem{defn}[thm]{Definition}
\theoremstyle{remark}

\newtheorem{rem}[thm]{Remark}
\numberwithin{thm}{section}
\numberwithin{equation}{section}

\newcommand{\cK}{\mathcal{K}}

\newcommand{\cO}{\mathcal{O}}

\newcommand{\Q}{\mathbf Q}

\newcommand{\Z}{\mathbf Z}
\newcommand{\bfT}{\mathbf T}

\newcommand{\rH}{{\rm H}}

\newcommand{\pp}{\mathfrak{p}}
\newcommand{\ppbar}{\overline{\mathfrak{p}}}
\newcommand{\frakp}{\mathfrak{p}}

\newcommand{\cI}{\mathcal{I}}
\newcommand{\bQ}{\mathbf{Q}}
\newcommand{\bZ}{\mathbf{Z}}
\newcommand{\bC}{\mathbf{C}}

\newcommand{\cR}{\mathbb{I}}

\newcommand{\F}{{\boldsymbol{f}}}
\newcommand{\G}{{\boldsymbol{g}}}

\DeclareMathOperator{\Hom}{Hom}

\DeclareMathOperator{\M}{M}

\newcommand{\pwseries}[1]{[[#1]]}
\def\k{\kappa}
\newcommand{\I}{\mathbb I}
\newcommand{\ro}{\mathfrak{O}_L}

\def\ac{{\rm ac}}
\def\cyc{{\rm cyc}}

\newcommand{\Tc}{{\mathbf{T}^{\dagger,{\rm ac}}}}

\newcommand{\unr}{{\hat{\bZ}_p^{\rm ur}}}

\newcommand{\fil}{\mathscr{F}}

\newcommand{\mat}[4]{\left(\begin{array}{cc}#1&#2\\#3&#4\end{array}\right)}

\begin{document}

\title
[The Iwasawa main conjectures for ${\rm GL}_2$ and derivatives of $p$-adic $L$-functions]{The Iwasawa main conjectures for ${\rm GL}_2$ and derivatives of $p$-adic $L$-functions}
\author[F.~Castella and X.~Wan]
{Francesc Castella and Xin Wan}
\address{Department of Mathematics, University of California, Santa Barbara, CA 93106, USA}
\email{castella@ucsb.edu}
\address{Morningside Center of Mathematics, Academy of Mathematics and Systems Science, Chinese Academy of Science, No.~55 Zhongguancun East Road, Beijing, 100190, China}
\email{xwan@math.ac.cn}

\thanks{During the preparation of this paper, F.C. was partially supported by the NSF grants DMS-{1801385} and DMS-{1946136}.}



\date{January 12, 2020.}



\begin{abstract}
We prove under mild hypotheses the three-variable Iwasawa main conjecture for $p$-ordinary modular forms in the indefinite setting. Our result is in a setting complementary to that in the work  of Skinner--Urban, and it has applications to Greenberg's nonvanishing conjecture for the first  derivatives at the center of $p$-adic $L$-functions of cusp forms in Hida families with root number $-1$ and to Howard's horizontal nonvanishing conjecture. 
%
\end{abstract}

\maketitle

\setcounter{tocdepth}{2}
\tableofcontents


\section{Introduction}


Fix a prime $p>3$ and a positive integer $N$ prime to $p$. Let $\ro$ be the ring of integers of a finite extension $L/\bQ_p$, and let 
\[
\F=\sum_{n=1}^\infty\boldsymbol{a}_nq^n\in\cR\pwseries{q}
\]
be a Hida family of tame level $N$, where $\cR$ is a finite flat extension
of the one-variable Iwasawa algebra $\ro\pwseries{T}$ with fraction field $F_\cR$. Throughout this paper, we shall assume that $\cR$ is regular. 
Let $\cK$ be an imaginary quadratic field of discriminant prime to $Np$, and
let $\Gamma_\cK:={\rm Gal}(\cK_\infty/\cK)$ be the Galois group of the $\bZ_p^2$-extension of $\cK$.

From work of Hida \cite{hidaII}, 
there is a $3$-variable $p$-adic $L$-function 
\[
L_p^{\tt Hi}(\F/\cK)\in\cR\pwseries{\Gamma_\cK}
\] 
interpolating critical values $L(\F_\phi/\cK,\chi,j)$ 
for the Rankin--Selberg $L$-function attached to the classical specializations 
of $\F$ twisted by
finite order characters $\chi:\Gamma_\cK\rightarrow\mu_{p^\infty}$. Let
\[
\rho_{\F}:G_{\bQ}:={\rm Gal}(\overline{\bQ}/\bQ)\rightarrow{\rm Aut}_{F_{\cR}}(V_{\F})\simeq{\rm GL}_2(F_\cR)
\]
be 
the Galois representation associated to $\F$ (which we take to be the contragredient of the Galois representation first constructed in \cite{hida86b}),  
and let $\bar{\rho}_{\F}:G_{\bQ}\rightarrow{\rm GL}_2(\kappa_\cR)$, where $\kappa_\cR=\cR/\mathfrak{m}_\cR$ is the residue field of $\cR$, be the associated semi-simple residual representation. 

By work of Mazur and Wiles \cite{MW-families,wiles88}, upon restriction to a decomposition group $D_p\subset G_{\bQ}$ at $p$ we have
\[
\bar{\rho}_{\F}\vert_{D_p}\sim
\left(\begin{smallmatrix}\bar{\varepsilon}& *\\& \bar{\delta}\end{smallmatrix}\right)
\]
where the character $\bar{\delta}$ is unramified. Under the assumption that
\begin{equation}\label{ass:irred}
\textrm{$\bar{\rho}_{\F}$ is irreducible and $\bar\varepsilon\neq\bar{\delta}$},\tag{MT}
\end{equation}
one knows that there exists a $G_\bQ$-stable lattice $T_{\F}\subset V_{\F}$ which is free of rank two over $\cR$ with the inertia coinvariants $\fil^-T_\F$ being $\cR$-free of rank one. Set
\[
A_\F:=T_\F\otimes_\cR\cR^\vee,\quad
\fil^-A_\F:=(\fil^-T_\F)\otimes_\cR\cR^\vee,
\]
where $\cR^\vee:={\rm Hom}_{\rm cts}(\cR,\bQ_p/\bZ_p)$ is the Pontryagin dual of $\cR$, and define the Greenberg Selmer group ${\rm Sel}_{\tt Gr}(\cK_\infty,A_\F)$ by
\begin{equation}\label{eq:2Gr-intro}
{\rm Sel}_{\tt Gr}(\cK_\infty,A_\F):=\ker\biggl\{\rH^1(\cK_\infty,A_\F)\rightarrow
\prod_{w\nmid p}\rH^1(I_w,A_\F)\times\prod_{w\vert p}\rH^1(\cK_{\infty,w},\fil^-A_\F)\biggr\},
\end{equation}
where $w$ runs over the places of $\cK_\infty$. The Pontryagin dual
\[
X_{\tt Gr}(\cK_\infty,A_\F):={\rm Hom}_{\rm cts}({\rm Sel}_{\tt Gr}(\cK_\infty,A_\F),\bQ_p/\bZ_p)
\]
is easily seen to be a finitely generated $\cR\pwseries{\Gamma_\cK}$-module. 

In this paper we study the following instance of the Iwasawa--Greenberg main conjectures (see \cite{Greenberg55}).

\begin{intro-conj1}[{\bf Iwasawa--Greenberg main conjecture}]
The module $X_{\tt Gr}(\cK_\infty,A_\F)$ is $\cR[[\Gamma_\cK]]$-torsion, and 
\[
{\rm Char}_{\cR[[\Gamma_\cK]]}(X_{\tt Gr}(\cK_\infty,A_{\F}))=(L_p^{\tt Hi}(\F/\cK))
\]
as ideals in $\cR[[\Gamma_\cK]]$.
\end{intro-conj1}

Many cases of this conjecture are known by the work of Skinner--Urban \cite{SU} and \cite{Kato295}. As we shall explain below, in this paper we place ourselves in a setting complementary to that in \cite{SU}, obtaining the following new result towards the Iwasawa--Greenberg main conjecture. The imaginary quadratic field $\cK$ determines a factorization
\[
N=N^+N^-
\]
with $N^-$ being the largest factor of $N$ divisible only by primes inert in $\cK$. 

\begin{ThmA}
In addition to {\rm (\ref{ass:irred})}, assume that:
\begin{itemize}
\item{} $N$ is squarefree,
\item{} some specialization $\F_\phi$ is the $p$-stabilization of a newform $f\in S_2(\Gamma_0(N))$,
\item{} $N^-$ is the product of a positive even number of primes,
\item{} $\bar\rho_{\F}$ is ramified at every prime $q\vert N^-$,
\item{} $p$ splits in $\cK$.
\end{itemize}
Then $X_{\tt Gr}(\cK_\infty,A_\F)$ is $\cR[[\Gamma_\cK]]$-torsion, and 
\[
{\rm Char}_{\cR[[\Gamma_\cK]]}(X_{\tt Gr}(\cK_\infty,A_\F))=(L_p^{\tt Hi}(\F/\cK))
\]
as ideals in $\cR[[\Gamma_\cK]]\otimes_{\bZ_p}\bQ_p$.
\end{ThmA}



As in \cite{SU}, the fact that $X_{\tt Gr}(\cK_\infty,A_\F)$ is $\cR[[\Gamma_\cK]]$-torsion
follows easily from Kato's work, and the proof of Theorem~A is reduced to establishing the divisibility ``$\subset$'' as ideals in $\cR[[\Gamma_\cK]]$ predicted by the main conjecture. For the proof of this divisibility, 
in \cite{SU} the authors study congruences between $p$-adic families of cuspidal automorphic forms and Eisenstein series on 
${\rm GU}(2,2)$,
and their method (in particular, their application of Vatsal's result \cite{vatsal-special}) relies crucially on their hypothesis that $N^-$ is the squarefree product of an \emph{odd} number of primes. 

In contrast, for the proof of Theorem~A we first link the above main conjecture with another instance of the Iwasawa--Greenberg main conjectures, and exploit our assumption on $N^-$ to prove the latter using Heegner points. 

As a consequence of our approach, we also obtain 
an application to Greenberg's conjecture (see \cite[\S{0}]{Nekovar-Plater}, following \cite{GreenbergCRM}) on the generic order of vanishing at the center of the $p$-adic $L$-functions attached to cusp form in Hida families. To state this, assume for simplicity that $\cR$ is just $\ro[[T]]$, and for each $k\in\bZ_{\geqslant 2}$ let $\F_k$ be the $p$-stabilized newform on $\Gamma_0(Np)$ obtained by setting $T=(1+p)^{k-2}-1$ in $\F$. One can show that the $p$-adic $L$-functions $L_p^{\tt MTT}(\F_k,s)$ of \cite{mtt} satisfy a functional equation 
\[
L_p^{\tt MTT}(\F_k,s)=-wL_p^{\tt MTT}(\F_\k,k-s)
\]
with a sign $w=\pm{1}$ independent of $k\in\mathbf{Z}_{\geqslant 2}$ with $k\equiv 2\pmod{2(p-1)}$. 

\begin{intro-conj4}[{\bf Greenberg's nonvanishing conjecture}]
Let $e\in\{0,1\}$ be such that $-w=(-1)^{e}$. Then 
\[
\frac{L_p^{\tt MTT}(\F_k,s)}{(s-k/2)^{e}}\biggr\vert_{s=k/2}\neq 0,
\]
for all but finitely many $k\in\bZ_{\geqslant 2}$ with $k\equiv 2\pmod{2(p-1)}$.
\end{intro-conj4} 

In other words, for all but finitely many $k$ as above, the order of vanishing of $L_p^{\tt MTT}(\F_k,s)$ at the center should be the least allowed by the sign in the functional equation.

To state our result in the direction of this conjecture, let 
\[
T_\F^\dagger:=T_\F\otimes\Theta^{-1}
\]
be the self-dual twist of $T_\F$. By work of Plater \cite{Plater} (and more generally Nekov{\'a}{\v{r}}  \cite{nekovar310}) 
there is a cyclotomic  $\cR$-adic height pairing 
\begin{equation}\label{eq:ht-intro}
\langle,\rangle_{\cK,\cR}^{\rm cyc}:{\rm Sel}_{\tt Gr}(\cK,T_\F^\dagger)\times
{\rm Sel}_{\tt Gr}(\cK,T_\F^\dagger)\rightarrow F_\cR
\end{equation} 
interpolating the $p$-adic height pairings for the classical specialization of $\F$ 
as constructed by Perrin-Riou \cite{PR-109}. 
It is expected that $\langle,\rangle_{\cK,\cR}^{\rm cyc}$ is non-degenerate, in the sense that its kernel on either side should reduce to $\cR$-torsion submodule of ${\rm Sel}_{\tt Gr}(\cK,T_\F^\dagger)$. 

\begin{ThmB}
In addition to {\rm (\ref{ass:irred})}, assume that:
\begin{itemize}
\item $N$ is squarefree,
\item 
$\F_2$ is old at $p$,
\item there are at least two primes $\ell\Vert N$ at which $\bar{\rho}_{\F}$ is ramified.
\end{itemize}
If ${\rm Sel}_{\tt Gr}(\bQ,T_{\F}^\dagger)$ has $\cR$-rank one and
$\langle,\rangle_{\cK,\cR}^{\rm cyc}$ is non-degenerate, then
\[
\frac{d}{ds}L_p^{\tt MTT}(\F_k,s)\biggr\vert_{s=k/2}\neq 0,
\]
for all but finitely many $k\in\bZ_{\geqslant 2}$ with $k\equiv 2\pmod{2(p-1)}$.
\end{ThmB}	
	


\begin{rem}
The counterpart to Theorem~B in rank zero, i.e., the implication
\begin{equation}\label{eq:rank0}
{\rm rank}_{\cR}\;{\rm Sel}_{\tt Gr}(\bQ,T_{\F}^\dagger)=0\quad\Longrightarrow \quad L_p^{\tt MTT}(\F_k,k/2)\neq 0,
\end{equation}
for all but finitely many $k$ as above, 
follows easily from 
\cite{SU} (see Theorem~\ref{thm:Gr+1}). 
\end{rem}

\begin{rem}\label{rem:0or1}
By the control theorem for ${\rm Sel}_{\tt Gr}(\bQ,T_{\F}^\dagger)$ (see e.g. \cite[Prop.~12.7.13.4(i)]{nekovar310}) 
and the $p$-parity conjecture  for classical Selmer groups (see e.g. \cite[Thm.~6.4]{cas-hsieh1}), the hypothesis that ${\rm Sel}_{\tt Gr}(\bQ,T_{\F}^\dagger)$ has $\cR$-rank one (resp. zero) implies that $w=1$ (resp. $w=-1$). 

Conversely, it is expected that the $\cR$-rank of ${\rm Sel}_{\tt Gr}(\bQ,T_{\F}^\dagger)$ is \emph{always} $0$ or $1$; more precisely,
\[
{\rm rank}_{\cR}\;{\rm Sel}_{\tt Gr}(\bQ,T_{\F}^\dagger)\overset{?}=\left\{
\begin{array}{ll}
1&\textrm{if $w=1$,}\\[0.1cm]
0&\textrm{if $w=-1$.}
\end{array}
\right. 
\]
For example, by \cite[Cor.~3.4.3 and Eq.~(21)]{howard-invmath} this prediction is a consequence of Howard's ``horizontal nonvanishing conjecture''.
\end{rem}

\begin{rem}\label{rem:CM-case}
For Hida families $\F$ with CM (a case that is excluded by our hypotheses), the analogue of Theorem~B is due to Agboola--Howard and Rubin \cite[Thm.~B]{AHsplit}. 
In this case, the rank one and non-degeneracy assumptions 
follow from Greenberg's nonvanishing results \cite{greenberg-BSD} (see \cite[Prop.~2.4.4]{AHsplit}) 
and a transcendence result of Bertrand \cite{bertrand-AH} (see \cite[Thm.~A.1]{AHsplit}). In rank zero, the CM case of $(\ref{eq:rank0})$ follows from \cite{greenberg-BSD} and Rubin's proof of the Iwasawa main conjecture for imaginary quadratic fields \cite{rubin-IMC}.
\end{rem}

We conclude this Introduction with some of the ingredients that go into the proofs of the above results. 

The proof of Theorem~A builds on the link that we establish in $\S\ref{sec:Iw}$ between different instances of the Iwasawa--Greenberg main conjectures involving Selmer groups differing from $(\ref{eq:2Gr-intro})$ in their local conditions at the places above $p$. In particular, letting $\pp$ be the prime of $\cK$ above $p$ determined by a fixed embedding $\overline{\bQ}\hookrightarrow\overline{\bQ}_p$,  and denoting by $\hat\bZ_p^{\rm ur}$ the completion of the ring of integers of the maximal unramified extension of $\bQ_p$, a central role is played by the Selmer group 
defined by
\[
{\rm Sel}_{\emptyset,0}(\cK_\infty,A_\F):=\ker\biggl\{\rH^1(\cK_\infty,A_\F)\rightarrow\prod_{w\nmid p}\rH^1(I_w,A_\F)\times
\prod_{w\mid\overline{\pp}}\rH^1(\cK_{\infty,w},A_\F)\biggr\}
\]
whose Pontryagin dual is conjecturally generated by a $p$-adic $L$-function
\[
\mathscr{L}_\pp(\F/\cK)\in\cR^{\rm ur}[[\Gamma_\cK]],\quad\quad\textrm{where}\;\;\cR^{\rm ur}:=\cR\hat{\otimes}_{\bZ_p}\hat{\bZ}_p^{\rm ur},
\] 
interpolating critical values $L(\F_\phi/\cK,\chi,j)$
with $\chi$ running over characters of $\Gamma_\cK$ with associated theta series of weight higher than the weight of $\F_\phi$. 

This second instance of the Iwasawa--Greenberg main conjecture can be related on the one hand to the main conjecture for $L_p^{\tt Hi}(\F/\cK)$ by building on the explicit reciprocity laws for the Rankin--Eisenstein classes of Kings--Loeffler--Zerbes \cite{KLZ2}, and on the other hand (after restriction to the anticyclotomic line) to the big Heegner point main conjecture of Howard \cite[Conj.~3.3.1]{howard-invmath}
using the explicit reciprocity law for Heegner points of \cite{cas-hsieh1}\footnote{Itself a generalization of the celebrated $p$-adic Waldspurger formula of Bertolini--Darmon--Prasanna \cite{bdp1}.}, thereby allowing us to take the results of \cite{wanIMC} and \cite{Fouquet} towards the proof of those conjecture to bring to bear on the main conjecture for $L_p^{\tt Hi}(\F/\cK)$. 

On the other hand, a key ingredient in the proof of Theorem~B is the Birch and Swinnerton-Dyer type formula for  
$L_p^{\tt Hi}(\F/\cK)$ along $\cR[[\Gamma_\cK^\ac]]$ that we obtain in Theorem~\ref{thm:3.1.5} by building on the ealier results of the paper, leading to a $p$-adic Gross--Zagier formula for Howard's system of big Heegner points $\mathfrak{Z}_\infty$ that we then apply for a suitably chosen imaginary quadratic field $\cK$.

\sk 

\noindent{\bf Acknowledgements.} 
It is a pleasure to thank Chris Skinner for several helpful conversations. 
Subtantial progress on this paper occurred during visits of the first author to Fudan University in January 2019, the Morningside Center of Mathematics in June 2019, and Academia Sinica in December 2019, and he would like to thank these institutions 
for their hospitality.

\section{$p$-adic $L$-functions}\label{sec:padicL}

\subsection{Hida families}\label{subsec:hida}

Let $\cR$ be a local reduced normal extension of $\ro[[T]]$, where $\ro$ is the ring of integers of a finite extension $L$ of $\bQ_p$, and denote by $\mathcal{X}_a(\cR)\subset{\rm Hom}_{\rm cts}(\cR,\overline{\bQ}_p)$ the set of continuous $\ro$-algebra homomorphisms
$\phi:\cR\rightarrow\overline{\bQ}_p$ satisfying 
\[
\phi(1+T)=\zeta(1+p)^{k-2}
\]
for some $p$-power root of unity $\zeta=\zeta_\phi$ and some integer $k=k_\phi\in\bZ_{\geqslant 2}$ called the \emph{weight} of $\phi$. We shall refer to the elements of $\mathcal{X}_a(\cR)$ as \emph{arithmetic} primes of $\cR$, and let $\mathcal{X}_a^o(\cR)$ denote the set consisting of arithmetic primes $\phi$ with $\zeta_\phi=1$ and weight $k_\phi\equiv 2\pmod{p-1}$.

Let $N$ be a positive integer prime to $p$, let $\chi$ be an even Dirichlet character modulo $Np$ taking values in $L$, and let $\F=\sum_{n=1}^\infty\boldsymbol{a}_nq^n\in\cR[[q]]$ be an ordinary $\cR$-adic cusp eigenform of tame level $N$ and character $\chi$, as defined in \cite[\S{3.3.9}]{SU}. In particular, for every $\phi\in\mathcal{X}_a(\cR)$ we have
\[
\F_\phi:=\sum_{n=1}^\infty\phi(\boldsymbol{a}_n)q^n\in S_{k}(\Gamma_0(p^{t}N),\chi\omega^{2-k_\phi}\psi_{\zeta}),
\]
where 
\begin{itemize}
\item $t=t_\phi\geqslant 1$ is such that $\zeta$ is a primitive $p^{t-1}$-st root of unity,  
\item $\omega$ is the Teichm\"uller character, and
\item
$\psi_{\zeta}:(\bZ/p^{t}N\bZ)^\times\twoheadrightarrow(\bZ/p^{t}\bZ)^\times\rightarrow\overline{\bQ}_p^\times$ is determined by $\psi_{\zeta}(1+p)=\zeta=\zeta_\phi$. 
\end{itemize}
Denote by $S^{\rm ord}(N,\chi;\cR)$ the space of such $\cR$-adic eigenforms $\F$. If in addition $\F_\phi$ is $N$-new for all $\phi\in\mathcal{X}_a(\cR)$, we say that $\F$ is a \emph{Hida family} of tame level $N$ and character $\chi$. 

We refer to $\F_\phi$ as the specialization of $\F$ at $\phi$. More generally, if $\phi\in{\rm Hom}_{\rm cts}(\cR,\overline{\bQ}_p)$ is such that $\F_\phi$ is a classical eigenform, we say that $\F_\phi$ is a classical specialization of $\F$; this includes the specializations of $\F\in S^{\rm ord}(N,\chi;\cR)$ at $\phi\in\mathcal{X}_a(\cR)$, but possibly also specializations in weight $1$, for example.


\subsection{Congruence modules}\label{subsec:congr}

We recall the notion of congruence modules following the treatment of \cite[$\S{12.2}$]{SU} and \cite[\S{3.3}]{hsieh-triple}. Let $\F$ be a Hida family of tame level $N$ and character $\chi$ defined over $\cR$, and let $\rho_\F:G_\bQ\rightarrow{\rm GL}_2(F_\cR)$ be the Galois representation associated to $\F$, where $F_\cR$ is the fraction field of $\cR$. Let $\mathbb{T}(N,\chi,\cR)$ be the Hecke algebra acting $S^{\rm ord}(N,\chi;\cR)$, and let $\lambda_{\F}:\mathbb{T}(N,\chi,\cR)\rightarrow\cR$ be the algebra homomorphism defined by $\cR$, which factors through the local component $\mathbb{T}_{\mathfrak{m}_{\F}}$.

Since $\F$ is $N$-new, there is an algebra direct sum decomposition
\[
\lambda:\mathbb{T}_{\mathfrak{m}_{\F}}\otimes_\cR F_\cR\simeq\mathbb{T}'\times F_\cR
\]
with the projection onto the second factor given by $\lambda_{\F}$. The \emph{congruence module} $C(\F)\subset\cR$ is defined by
\[
C(\F):=\lambda_{\F}\left(\mathbb{T}_{\mathfrak{m}_{\F}}\cap\lambda^{-1}(\{0\}\times F_\cR)\right).
\]

Following the convention in  \cite[\S{7.7}]{KLZ2}, we shall also consider the \emph{congruence ideal} $I_\F$, defined as the fractional ideal $I_\F:=C(\F)^{-1}\subset F_\cR$. As noted in \emph{loc.cit.}, if follows from \cite[Thm.~4.2]{hidaII} that elements of $I_\F$ define meromorphic functions on ${\rm Spec}(\cR)$ which are regular at all arithmetic points. 


\subsection{Rankin--Selberg $p$-adic $L$-functions}\label{sec:2varL}
 
The next result on the construction of $3$-variable $p$-adic Rankin $L$-series is due to Hida.  

Let $\Gamma$ 
be Galois group of the cyclotomic $\bZ_p^\times$-extension of $\bQ$, and set 
\[
\Lambda_\Gamma=\bZ_p[[\Gamma]].
\] 
If $j\in\bZ$ and $\chi$ is a Dirichlet character of $p$-power conductor, there is a unique $\phi\in{\rm Hom}_{\rm cts}(\Lambda_\Gamma,\overline{\bQ}_p^\times)$ extending the character $z\mapsto z^j\chi(z)$ on $\bZ_p^\times$.

\begin{thm}\label{thm:hida}
Let $\F_1, \F_2$ be Hida families of tame levels $N_1, N_2$, respectively, and let $N={\rm lcm}(N_1,N_2)$. Then there is an element
\[
L_p(\F_1,\F_2)\in\left(I_{\F_1}\hat{\otimes}_{\bZ_p}\cR_{\F_2}\hat{\otimes}_{\bZ_p}\Lambda_{\Gamma}\right)\otimes_{\bZ}\bZ[\mu_N]
\]
uniquely characterized by the following interpolation property. Let $f_1$, $f_2$ be classical specializations of $\F_1$, $\F_2$ of weights $k_1$, $k_2$, respectively, with $k_1>k_2\geqslant 1$, let $j$ be an integer in the range $k_2\leqslant j\leqslant k_1-1$, and let $\chi$ be a Dirichlet character of $p$-power conductor.
Suppose the automorphic representation $\pi_{f_1}$ is a principal series representation $\pi(\eta_1,\eta_1')$ with $\eta_1$ unramified and $\eta_1(p)$ a $p$-adic unit.
Then the value of $L_p(\F_1,\F_2)$ at the corresponding specialization $\phi\in{\rm Spec}(\cR_{\F_1}\hat\otimes_{\bZ_p}\cR_{\F_2}\hat\otimes_{\bZ_p}\Lambda_{\Gamma})$ is given by
\begin{align*}
\phi(L_p(\F_1,\F_2))&=\frac{\mathcal{E}(f_1,f_2,\chi,j)}{\mathcal{E}(f_1)\mathcal{E}^*(f_1)}\cdot\frac{\Gamma(j)\Gamma(j-k_2+1)}{\pi^{2j+1-k_2}(-i)^{k_1-k_2}2^{2j+k_1-k_2}\left\langle f_1,f_1^c\vert_{k_1}\bigl(\begin{smallmatrix}&-1\\p^{t_1}N_1\end{smallmatrix}\bigr)\right\rangle_{N_1}}\\
&\quad\times L(
f_1,f_2,\chi^{-1},j),
\end{align*}
where if $\alpha_{i}$ and $\beta_{i}$ are the roots of the Hecke polynomial of $f_i$ at $p$, with $\alpha_{i}$ being the $p$-adic unit root, and $p^t$ is the conductor of $\chi$, the Euler factors are given by
\[
\mathcal{E}(f_1,f_2,\chi,j)=\left\{
\begin{array}{ll}
\left(1-\frac{p^{j-1}}{\alpha_{1}\alpha_{2}}\right)\left(1-\frac{p^{j-1}}{\alpha_{1}\beta_{2}}\right)
\left(1-\frac{\beta_{1}\alpha_{2}}{p^j})(1-\frac{\beta_{1}\beta_{2}}{p^j}\right)&\textrm{if $t=0$,}\\[0.2cm]
G(\chi)^2\cdot\left(\frac{p^{2j-2}}{\alpha^2_1\alpha_2\beta_2}\right)^t &\textrm{if $t\geqslant 1$,}
\end{array}
\right.
\]
where $G(\chi)$ is the Gauss sum of $\chi$, and if $p^{t_1}$ is the $p$-part of the conductor of $\eta_1'$, then
\[
\mathcal{E}(f_1)\mathcal{E}^*(f_1)=\left\{
\begin{array}{ll}
\left(1-\frac{\beta_{1}}{p\alpha_{1}}\right)\left(1-\frac{\beta_{1}}
{\alpha_{1}}\right)&\textrm{if $t_1=0$,}\\[0.2cm]
G(\chi_1)\cdot\eta_1'\eta_1^{-1}(p^{t_1})p^{-t_1}&\textrm{if $t_1\geqslant 1$,}\\
\end{array}
\right.
\]
where $\chi_1$ is the nebentypus of $f_1$.
\end{thm}

\begin{proof}
This follows from \cite[Thm.~5.1]{hidaII}, which we have stated adopting the formulation in \cite[Thm.~7.7.2]{KLZ2} (slightly extended to include some more general specializations of the dominant Hida family $\F_1$).
\end{proof}



In this paper, we shall consider the $p$-adic $L$-functions  $L_p(\F_1,\F_2)$ of Theorem~\ref{thm:hida} in the cases where either $\F_1$ or $\F_2$ has CM. 

Thus let $\F$ be a fixed Hida family of tame level $N$ defined over $\cR$, and let $\cK$ be an imaginary quadratic field of discriminant $-D_{\cK}<0$ prime to $pN$ such that
\begin{equation}\label{eq:spl}
\textrm{$p=\frakp\overline{\pp}$ splits in $\cK$,}\nonumber
\end{equation}
with $\pp$ denoting the prime of $\cK$ above $p$ induced by our fixed embedding   
$\imath_p:\overline\bQ\hookrightarrow\bC_p$. 

Let $\cK_\infty$ be the $\bZ_p^2$-extension of $\cK$, and denote by $\Gamma_\pp\simeq\bZ_p$ the Galois group over $\cK$ of the maximal subfield of $\cK_\infty$ unramified outside $\overline{\pp}$. Let 
\begin{equation}\label{eq:g-CM}
\G=\sum_{n=1}^\infty\boldsymbol{b}_nq^n
\in\cR_{\G}[[q]]
\end{equation}
be the canonical Hida family of CM forms constructed in \cite[\S{5.2}]{JSW}, where $\cR_{\G}=\bZ_p[[\Gamma_\pp]]$. Specifically, denoting by $\theta_\pp:\mathbb{A}_{\cK}^\times\rightarrow\Gamma_\pp$ the composition of the global reciprocity map ${\rm rec}_\cK:\mathbb{A}_\cK^\times\rightarrow G_\cK^{\rm ab}$ with the natural projection $G_\cK^{\rm ab}\twoheadrightarrow\Gamma_\pp$, we have  
\[
\boldsymbol{b}_n=\sum_{\substack{N(\mathfrak{a})=n, (\mathfrak{a},\overline{\pp})=1}}\theta_\pp(x_{\mathfrak{a}}), 
\]
summing over integral ideals $\mathfrak{a}\subset\cO_\cK$, and $x_\mathfrak{a}\in\mathbb{A}_{\cK}^{\infty,\times}$ is any finite id\`ele of $\cK$ with ${\rm ord}_w(x_{\mathfrak{a},w})={\rm ord}_{w}(\mathfrak{a})$ for all finite places $w$ of $\cK$.

\subsection{Non-dominant CM: $\F_2=\G$}\label{subsec:L-2}\label{sec:dom-CM}

Assume that the residual representation $\bar{\rho}_{\F}$ is irreducible and $p$-distinguished. Then by \cite[Cor.~2, p,~482]{Fermat-Wiles} the local ring $\mathbb{T}_{\mathfrak{m}_{\F}}$ introduced in $\S\ref{subsec:congr}$ is Gorenstein, and by Hida's results (see e.g. 
\cite{hida-AJM88}) the congruence module $C(\F)$ is principal. Let $c_\F\in C(\F)$ be a generator, and set
\[
L_p^{\tt Hi}(\F/\cK):=c_{\F}\cdot L_p(\F,\G),
\]
viewed as an element in $\cR[[\Gamma_\cK]]$ (well-defined up to a unit in $\cR^\times$), where $\Gamma_\cK={\rm Gal}(\cK_\infty/\cK)$. 

The action of complex conjugation yields a decomposition
\[
\Gamma_\cK\simeq
\Gamma_\cK^{\rm ac}\times\Gamma_\cK^{\rm cyc},
\]
where 
$\Gamma_\cK^{\rm ac}$ (resp. $\Gamma_\cK^{\rm cyc}$) denotes the Galois group of the anticyclotomic (resp. cyclotomic) $\bZ_p$-extension of $\cK$. We next study the projections of $L_p^{\tt Hi}(\F/\cK)$ to $\cR[[\Gamma_\cK^\ac]]$ and $\cR[[\Gamma_\cK^{\rm cyc}]]$. 

\subsubsection{Anticyclotomic restriction of $L_p^{\tt Hi}(\F/\cK)$}\label{sec:anti-L}

Assume that $\F$ has trivial tame character, and following  \cite[Def.~2.1.3]{howard-invmath} define the \emph{critical character} $\Theta:G_\bQ\rightarrow\cR^\times$ by
\begin{equation}\label{def:crit-char}
\Theta:=[\langle\varepsilon_{\rm cyc}\rangle^{1/2}],
\end{equation}
where $\varepsilon_{\rm cyc}:G_\bQ\rightarrow\bZ_p^\times$ is the cyclotomic character, $\langle\cdot\rangle:\bZ_p^\times\rightarrow 1+p\bZ_p$ is the natural projection, and 
\[
[\cdot]:1+p\bZ_p\hookrightarrow\bZ_p[[1+p\bZ_p]]^\times\simeq\bZ_p[[T]]^\times\rightarrow\cR^\times
\]
is the composition of the obvious maps. This induces the twist map
\begin{equation}\label{eq:tw-theta}
{\rm tw}_{\Theta^{-1}}:\cR[[\Gamma_\cK]]\rightarrow\cR[[\Gamma_\cK]]
\end{equation}
defined by $\gamma\mapsto\Theta^{-1}(\gamma)\gamma$ for $\gamma\in\Gamma_\cK$.

Write $N$ as the product
\[
N=N^+ N^-
\]
with $N^+$ (resp. $N^-$) divisible only by primes which are split (resp. inert) in $\cK$, and consider the following generalized \emph{Heegner hypothesis}: 
\begin{equation}\label{eq:gen-Heeg-f}
\textrm{$N^-$ is the squarefree product of an even number of primes.}\tag{gen-H}
\end{equation}
Whenever we assume that $\cK$ satisfies the hypothesis (\ref{ass:gen-H}), we fix an integral ideal $\mathfrak{N}^+\subset\cO_\cK$ with $\cO_\cK/\mathfrak{N}^+\simeq\bZ/N^+\bZ$.


\begin{prop}\label{thm:hida-1}
Let $L_p^{\tt Hi}(\F^\dagger/\cK)_{\ac}$ be the image of ${\rm tw}_{\Theta^{-1}}(L_p^{\tt Hi}(\F/\cK))$ under the natural projection $\cR[[\Gamma_\cK]]\rightarrow\cR[[\Gamma_\cK^{\ac}]]$. If $\cK$ satisfies the hypothesis {\rm (\ref{eq:gen-Heeg-f})}, 
then $L_p^{\tt Hi}(\F^\dagger/\cK)_{\ac}$ is identically zero.
\end{prop}

\begin{proof}
Let $\phi\in{\rm Spec}(\cR_{\F}\hat\otimes_{\bZ_p}\cR_{\G}\hat\otimes_{\bZ_p}\Lambda_\Gamma)={\rm Spec}(\cR[[\Gamma_\cK]])$ be a specialization in the range specified in Theorem~\ref{thm:hida}, with $f_1=\F_\phi$ the $p$-stabilization of a newform $f\in S_k(\Gamma_0(N))$ of weight $k\geqslant 2$ and $f_2=\G_\phi$ a classical weight $1$ specialization. 

By the interpolation property, the value $\phi(L_p^{\tt Hi}(\F/\cK))$ is a multiple of
\[
L(f_1,f_2,\chi^{-1},j)=L(f/\cK,\psi,j),
\]
with $\psi$ a finite order character of $\Gamma_\cK$ and $1\leqslant j\leqslant k-1$, and so $\phi({\rm tw}_{\Theta^{-1}}(L_p^{\tt Hi}(\F/\cK)))$ is also a multiple $L(f/\cK,\psi',k/2)$ for a finite order character $\psi'$ of $\Gamma_\cK$. If $\psi'$ factors through the projection $\Gamma_\cK\twoheadrightarrow\Gamma_{\cK}^{\ac}$, then the $L$-function $L(f/\cK,\psi',s)$ is self-dual, with
a functional equation relating its values at $s$ and $k-s$, and if $\cK$ satisfies the hypothesis (\ref{eq:gen-Heeg-f}), then the sign in this
functional equation is $-1$ (see e.g. \cite[\S{1}]{CV-dur}). Thus $L(f/\cK,\psi',k/2)=0$, and letting $\phi$ vary, the result follows.
\end{proof}

\subsubsection{Cyclotomic restriction of $L_p^{\tt Hi}(\F/\cK)$}\label{sec:anti-L}

As above, we denote by $\Gamma_\cK^{\rm cyc}$ the Galois group of the cyclotomic $\bZ_p$-extension of $\cK$, which we shall often identify with the maximal torsion-free quotient of $\Gamma$. 

For any ordinary $p$-stabilized newform $f$ of tame level $N$ defined over $L$ (a finite extension of $\bQ_p$), let $L_p^{\tt MTT}(f)\in\ro[[\Gamma]]$ be the cyclotomic $p$-adic $L$-function attached to $f$ in \cite{mtt}, where $\ro$ is the ring of integers of $L$ (see \cite[\S{3.4.4}]{SU} and the references therein). We refer the reader to \emph{loc.cit.} for the precise interpolation property satisfied by $L_p^{\tt MTT}(f)$, only noting here that the complex periods used for the construction are Shimura's periods $\Omega_f^\pm\in\bC^\times/\ro^\times$ (as reviewed in \cite[\S{3.3.3}]{SU}). 


\begin{thm}\label{thm:cyc-res}
Let $L_p^{\tt Hi}(\F/\cK)_{\rm cyc}$ be the image of $L_p^{\tt Hi}(\F/\cK)$ under the natural projection $\cR[[\Gamma_\cK]]\rightarrow\cR[[\Gamma_\cK^{\rm cyc}]]$. Then for every $\phi\in\mathcal{X}_a^o(\cR)$ 
we have
\[
\phi(L_p^{\tt Hi}(\F/\cK)_{\rm cyc})=L_p^{\tt MTT}(\F_\phi)\cdot L_p^{\tt MTT}(\F_\phi\otimes\epsilon_\cK)
\]
up to a unit in $\phi(\cR)[[\Gamma]]^\times$, where $\epsilon_\cK$ is the quadratic character associated to $\cK$.
\end{thm}

\begin{proof}
Since we assume that $\bar{\rho}_{\F}$ satisfies the hypotheses (\ref{ass:irred}), by \cite[Thm.~0.1]{hida-AJM88} (see also \cite[Lem.~12.1]{SU}) for every $\phi\in\mathcal{X}_a^o(\cR)$ we have the relation 
\[
\langle\F_\phi,\F_\phi\rangle_{N}\cdot\phi(c_{\F})^{-1}=u\cdot\Omega_{\F_\phi}^\pm\cdot\Omega_{\F_\phi\otimes\epsilon_\cK}^\pm
\]
between the periods appearing in the interpolation property of the respective sides of the claimed equality, for some unit $u\in\phi(\cR)^\times$. Since by construction, $L_p^{\tt Hi}(\F/\cK)$ specializes at $\phi$ to the $p$-adic $L$-function $L_p(\F_\phi/\cK)$ considered 
in \cite{BL-ord}, 
the result thus follows from  \cite[Cor.~2.2]{BL-ord} and  \cite[Thm.~12.8]{SU}. 
\end{proof}

\subsection{Dominant CM: $\F_1=\G$}\label{subsec:L-1}\label{sec:nondom-CM}

As in $\S\ref{sec:dom-CM}$, let $\F\in\cR[[q]]$ be a fixed Hida family of tame level $N$, and let $\G$ be the CM Hida family in  (\ref{eq:g-CM}). 

Let $\unr$ be the completion of the ring of integers of the maximal unramified extension of $\bQ_p$, and set $\cR^{\rm ur}:=\cR\hat\otimes_{\bZ_p}\unr$. 

By \cite[\S{5.3.0}]{Katz49} (see also \cite[Thm.~II.4.14]{de-shalit}) there exists a $p$-adic $L$-function
$\mathscr{L}_{\pp}^{}(\cK)\in\unr[[\Gamma_\cK]]$
such that if $\psi$ is a character of $\Gamma_\cK$ corresponding to an algebraic Hecke character of $\cK$ with trivial conductor 
and infinity type $(a,b)$ with $0\leqslant-b<a$, then
\begin{equation}\label{eq:Katz}
\mathscr{L}^{}_{\pp}(\cK)(\psi)=\biggl(\frac{\sqrt{D_\cK}}{2\pi}\biggr)^{a}
\cdot\Gamma(b)\cdot(1-\psi(\pp))\cdot(1-p^{-1}\psi^{-1}(\overline\pp))
\cdot\frac{\Omega_p^{b-a}}{\Omega_K^{b-a}}\cdot L_{}(\psi,0),
\end{equation}
where $\Omega_K\in\bC^\times$ and $\Omega_p\in\bC_p^\times$ are certain CM periods (as defined in e.g. \cite[\S{2.5}]{cas-hsieh1}). 

Let $h_\cK$ be the class number of $\cK$, $w_\cK:=\vert\cO_\cK^\times\vert$, and set
\begin{equation}\label{eq:factor-hida}
\mathscr{L}_{\pp}(\F/\cK):=\biggl(\frac{h_\cK}{w_\cK}\mathscr{L}_\pp(\cK)_{\ac}\biggr)\cdot L_{p}(\G,\F),
\end{equation}
where $\mathscr{L}_\pp(\cK)_{\ac}$ is the anticyclotomic projection of $\mathscr{L}_\pp(\cK)$. A priori, $\mathscr{L}_\pp(\F/\cK)$ is an element in $I_\G\hat\otimes_{\bZ_p}\cR^{\rm ur}[[\Gamma_\cK]]$ but comparing its interpolating property with that of a different $3$-variable $p$-adic $L$-function, we can show its integrality. 


\begin{prop}
The $p$-adic $L$-function in {\rm (\ref{eq:factor-hida})} is integral, i.e., $\mathscr{L}_\pp(\F/\cK)\in\cR^{\rm ur}[[\Gamma_\cK]]$
\end{prop}

\begin{proof}
For any finite set $\Sigma$ of places $\cK$ outside $p$ and containing all the places dividing $N D_\cK$, the results of \cite[\S{7.5}]{wanIMC} yield the construction of the ``$\Sigma$-imprimitive'' element 
\[
\mathfrak{L}_\pp^\Sigma(\F/\cK)\in\cR^{\rm ur}[[\Gamma_\cK]]
\]
characterized by the following interpolation property.
For a Zariski dense set of arithmetic points $\phi\in{\rm Spec}(\cR[[\Gamma_\cK]])$ with $\F_\phi$  of weight $2$ and conductor $p^tN$	generating a unitary  $\pi_{\F_\phi}\simeq\pi(\chi_{1,p},\chi_{2,p})$ with $v_p(\chi_{1,p}(p))=-\frac{1}{2}$ and $v_p(\chi_{2,p}(p))=\frac{1}{2}$, and with $\boldsymbol{\psi}_\phi$ a Hecke character of $\cK$ of infinity type $(-n,0)$ for some $n\geqslant 3$ and conductor $p^t$, we have:	
\begin{align*}
\phi(\mathfrak{L}_{\pp}(\F/\cK))&=p^{(n-3)t}\boldsymbol{\psi}_{\phi,\pp}^2\chi_{1,p}^{-1}\chi_{2,p}^{-1}(p^{-t})G(\boldsymbol{\psi}_{\phi,\pp}\chi_{1,p}^{-1})G(\boldsymbol{\psi}_{\phi,\pp}\chi_{2,p}^{-1})\Gamma(n)\Gamma(n-1)\Omega_p^{2n}\\
&\quad\times \frac{L^\Sigma(\F_\phi,\chi^{-1}_{\phi}\boldsymbol{\psi}_\phi,0)}{(2\pi i)^{2n-1}\Omega_\cK^{2n}},
\end{align*}
where $\chi_{\phi}$ is the nebentypus of $\F_\phi$, $\Omega_p\in\bC_p^\times$ and $\Omega_\cK\in\bC^\times$ are CM periods, and $L^\Sigma(\F_\phi,\chi^{-1}_{\phi}\boldsymbol{\psi}_\phi,0)$ is the $\Sigma$-imprimitive Rankin--Selberg $L$-values. Setting
\begin{equation}\label{eq:prim}
\mathfrak{L}_\pp(\F/\cK):=\mathfrak{L}_\pp^\Sigma(\F/\cK)\times\prod_{w\in\Sigma}P_w(\Psi_\cK({\rm Frob}_w))^{-1},
\end{equation}
where $P_w$ is the Euler factor at $w$ and $\Psi_\cK:G_\cK\twoheadrightarrow\Gamma_\cK$ is the natural projection, we thus obtain an element interpolating the Rankin--Selberg $L$-values themselves, but which \emph{a priori} is just an element in the fraction field of $\cR^{\rm ur}[[\Gamma_\cK]]$. To see the inclusion $\mathfrak{L}_\pp(\F/\cK)\in\cR^{\rm ur}[[\Gamma_\cK]]$ we shall compare $(\ref{eq:prim})$ with the product in the right-hand side of $(\ref{eq:factor-hida})$; the required integrality of $\mathscr{L}_\pp(\F/\cK)$ will follow from this comparison. 

Any arithmetic point $\phi\in\mathrm{Spec}(\cR[[\Gamma_\cK]])$ as above can be written as the product $\boldsymbol{\psi}_\phi'\cdot\boldsymbol{\psi}_\phi''$, with $\boldsymbol{\psi}_\phi'$ cyclotomic (i.e., factoring through $\Gamma_\cK\twoheadrightarrow\Gamma_\cK^{\rm cyc}$), and $\boldsymbol{\psi}_\phi''$ corresponding to a Hecke character unramified at $\overline{\pp}$ and of infinity type $(-n,0)$. Then $\chi^{-1}_{\phi}\boldsymbol{\psi}'_\phi$ (resp. the theta series of $\boldsymbol{\psi}''_\phi$) corresponds to $\chi\vert\cdot\vert^j$ (resp. $f_1=\G_\phi$) in Theorem~\ref{thm:hida}, so that
\[ 
L(\F_\phi,\chi_{\phi}^{-1}\boldsymbol{\psi}_\phi,0)=L(f_1,f_2,\chi^{-1},j)
\]
with $f_2=\F_\phi$. Letting $p^tD_\cK$
be the conductor of $\G_\phi$, according to \cite[Thm.~7.1]{HT-ENS} a direct calculation shows that the product $\mathcal{E}(\G_\phi)\mathcal{E}^*(\G_\phi)\cdot\langle\G_\phi,\G_\phi\rangle_{p^tD_\cK}$ in Theorem~\ref{thm:hida} agrees with
\begin{equation}\label{eq:factor-RS}
\frac{\Gamma(n)G(\boldsymbol{\psi}''^{-1}_{\phi,\bar{\frakp}})L(\boldsymbol{\psi}''_\phi(\boldsymbol{\psi}''_\phi)^{-c}, 1)}{(-2\pi i)^n}\cdot\frac{L(\epsilon_\cK,1)}{-2\pi i}
\end{equation}
up to a $p$-adic unit independent of $\phi$, where $\epsilon_\cK$ is the quadratic character attached to $\cK$. 

By the class number formula, the second factor in $(\ref{eq:factor-RS})$ is given by $h_\cK$ up to a $p$-adic unit, while by the interpolation property of the Katz $p$-adic $L$-function $\mathscr{L}_\pp(\cK)$ (see \cite[\S{5.3.0}]{Katz49}), the left factor multiplied by $(\Omega_p/\Omega_\cK)^{2n}$ is interpolated, for varying $\phi$, by the anti-cyclotomic projection of $\mathscr{L}_\pp(\cK)$ viewed as an element in $\unr[[\Gamma_{\pp}]]$. This shows the factorization that $\mathscr{L}_\pp(\F/\cK)$ in $(\ref{eq:factor-hida})$ and $\mathfrak{L}_\pp(\F/K)$ differ by a unit.

Finally, by the proof of \cite[Prop.~8.3]{wan-combined}, the only possible denominators of $\mathscr{L}_\pp(\F/\cK)$ are powers of the augmentation ideal of $\bZ_p[[\Gamma_{\pp}]]$, while by $(\ref{eq:prim})$ the possible denominators can only be either powers of $p$ or factors coming from Euler factors at primes $w\in\Sigma$. Since these two sets are disjoint, integrality of $\mathfrak{L}_\pp(\F/\cK)$, and hence of $\mathscr{L}_\pp(\F/\cK)$, follows.
\end{proof}


We conclude this section by discussing the anticyclotomic restriction of $\mathscr{L}_\pp(\F/\cK)$, which contrary to $L_p^{\tt Hi}(\F/\cK)$ will be nonzero under the generalized Heegner hypothesis.

\begin{thm}\label{thm:bdp}
Assume 
that $\cK$ satisfies the hypothesis {\rm (\ref{eq:gen-Heeg-f})}, and if $N^->1$ assume in addition that $N$ is squarefree. Then there exists an element $\mathscr{L}_{\pp}^{\tt BDP}(\F/\cK)\in\cR^{\rm ur}[[\Gamma_\cK^{\ac}]]$ such that for every $\phi\in\mathcal{X}_a(\cR)$ of weight $k$ and trivial nebentypus, 
and every crystalline character $\psi$ of $\Gamma_\cK^\ac$ corresponding to a Hecke character of $\cK$ of infinity type $(n,-n)$ with $n\geqslant 0$, we have
\begin{align*}
\phi(\mathscr{L}_{\pp}^{\tt BDP}(\F/\cK)^2)(\psi)
&=\mathcal{E}_\pp(\F_\phi,\psi)^2\cdot
\psi(\mathfrak{N}^+)^{-1}\cdot 2^3\cdot\varepsilon(\F_\phi)\cdot w_\cK^2\sqrt{D_\cK}\cdot\Gamma(k+n)\Gamma(n+1)\Omega_p^{2k+4n}\\
&\quad\times\frac{L(\F_\phi/\cK,\psi,k/2)\cdot\alpha(\F_\phi,\F_\phi^B)^{-1}}{(2\pi)^{k+2n+1}\cdot({\rm Im}\;\boldsymbol{\theta})^{k+2n}\cdot\Omega_\cK^{2k+4n}},
\end{align*}
where $\mathcal{E}_\pp(\F_\phi,\psi)=(1-\phi(\boldsymbol{a}_p)\psi_{\overline\pp}(p)p^{-k/2})
(1-\phi(\boldsymbol{a}_p)^{-1}\psi_{\overline\pp}(p)p^{k/2-1})$, 
$\varepsilon(\F_\phi)$ is the global root number of $\F_\phi$, $w_\cK:=\vert\cO_\cK^\times\vert$, $\Omega_p\in\bC_p^\times$ and $\Omega_\cK\in\bC^\times$ are CM periods attached to $\cK$ as \cite[\S{2.5}]{cas-hsieh1}, $\boldsymbol{\theta}\in\cK$ is as in {\rm (\ref{eq:vartheta})} below, and  
\[
\alpha(\F_\phi,\F_\phi^B)=\frac{\langle\F_\phi,\F_\phi\rangle}{\langle\F_\phi^B,\F_\phi^B\rangle}
\]
is a ratio of Petersson norms of $\F_\phi$ and its transfer $\F_\phi^B$ to a quaternion algebra, normalized as in \cite[\S{2.2}]{prasanna}.
\end{thm}

\begin{proof}
When $N^-=1$, this is \cite[Thm.~2.11]{cas-2var} (in which case $\alpha(\F_\phi,\F_\phi^B)=1$). In the following   we sketch how to extend that result to include the more general hypothesis (\ref{ass:gen-H}). Some of the notations used here will be introduced later in $\S\ref{sec:HP}$.

Let $\cO_B$ be a maximal order of $B$, and let ${\rm Ig}_{N^+,N^-}$  be the Igusa scheme over $\bZ_{(p)}$ classifying abelian surfaces with $\cO_B$-multiplication and $U_\infty$-level structure (here $U_\infty$ is the open compact $U_r\subset\hat{R}_r^\times$ in $\S\ref{subsec:Sh}$ with $r=\infty$). For any valuation ring $W$ finite flat over $\bZ_p$, denote by $V_p(W)$ the module of formal functions on ${\rm Ig}_{N^+,N^-}$ (i.e., $p$-adic modular forms) defined over $W$, and set
\[
V_p(\cR):=V_p(W_0)\hat{\otimes}_{W_0}\cR,
\]
where $W_0=W(\kappa_\cR)$ is the ring of Witt vectors of the residue field of $\cR$. For every $\cO_\cK$-ideal $\mathfrak{a}$ prime to $\mathfrak{N}^+\pp$, the construction of $\varsigma^{(s)}$ (for arbitrary $s\geqslant 0$) in $\S\ref{subsec:construct}$ determines CM points $x(\mathfrak{a})\in{\rm Ig}_{N^+,N^-}$, and the argument in \cite[Thm.~3.2.16]{hida-GME} with the use of $q$-expansions and the $q$-expansion principle replaced by Serre--Tate-expansions and the resulting $t$-expansion principle around any such $x(\mathfrak{a})$ (see e.g \cite[p.~107]{hida-mu}) 
shows that every element $\F^B\in V_p(\cR)$ defines a $p$-adic family  (in fact, finite collections of such, since $\cR$ is finite over $W_0[[T]]$) of $p$-adic modular forms $\F^B_z=\F^B(u^z-1)\in V_p(W_0)$,  where $u=1+p$, indexed by $z\in\bZ_p$. 

The Hida family $\F$ corresponds to minimal prime in the localized universal $p$-ordinary Hecke algebra $\mathbb{T}_{\infty,\mathfrak{m}}^{\rm ord}$,  and by the integral Jacquet--Langlands correspondence (see e.g. the discussion in \cite[\S{5.3}]{LV}), there exists a $p$-adic family $\F_B$ as above corresponding to $\F$, which we normalize by requiring that some Serre--Tate expansion $\F_z^B(t)$ does not vanish modulo $p$. 

There are $U$- and $V$-operators acting on $\F_B$ defined as in \cite[\S{3.6}]{brooks}, and we set 
\[
\F_B^\flat:=\F_B\vert(VU-UV).
\]
With these, we may define $\cR^{\rm ur}$-valued measures $\mu_{\F_B,x(\mathfrak{a})}$ and $\mu_{\F_{B,\mathfrak{a}}^\flat}$ on $\bZ_p$ (with the latter supported on $\bZ_p^\times$ by \cite[Prop.~4.17]{brooks}) as in \cite[\S{2.7}]{cas-2var}, and define $\mathscr{L}_{\pp,\boldsymbol{\xi}}(\F)$ to be the $\cR^{\rm ur}$-valued measure on ${\rm Gal}(H_{p^\infty}/\cK)$ given by
\[
\mathscr{L}_{\pp,\boldsymbol{\xi}}(\F)(\phi)=
\sum_{[\mathfrak{a}]\in{\rm Pic}(\cO_\cK)}
\boldsymbol{\xi}\boldsymbol{\chi}^{-1}(\mathfrak{a})\mathbf{N}(\mathfrak{a})^{-1}
\int_{\bZ_p^\times}(\phi\vert[\mathfrak{a}])(z){\rm d}\mu_{\F^\flat_{B,\mathfrak{a}}}(z)
\]
for all $\phi:{\rm Gal}(H_{p^\infty}/\cK)\rightarrow\cO_{\bC_p}^\times$, where, if $\sigma_{\mathfrak{a}}$ corresponds to $\mathfrak{a}$ under the Artin reciprocity map, $\phi\vert[\mathfrak{a}]$ is the character on $z\in\bZ_p^\times$ given by $\phi(\sigma_\mathfrak{a}{\rm rec}_{{\pp}}(z))$ for the local reciprocity map ${\rm rec}_\pp:\cK_\pp^\times\rightarrow G_\cK^{\rm ab}\rightarrow\Gamma_\cK^\ac$, $\boldsymbol{\chi}:\cK^\times\backslash\mathbb{A}_\cK^\times\rightarrow\cR^\times$ is the character given by $x\mapsto\Theta({\rm rec}_\bQ({\rm N}_{\cK/\bQ}(x)))$ for the reciprocity map ${\rm rec}_\bQ:\bQ^\times\backslash\mathbb{A}^\times\rightarrow G_\bQ^{\rm ab}$, and $\boldsymbol{\xi}$ is the auxiliary anticyclotomic $\cR$-adic character constructed in \cite[Def.~2.8]{cas-2var}. 

Still denoting by $\mathscr{L}_{\pp,\boldsymbol{\xi}}(\F)$ its image under the natural projection $\cR^{\rm ur}[[{\rm Gal}(H_{p^\infty}/\cK)]]\rightarrow\cR^{\rm ur}[[\Gamma_\cK^{\rm ac}]]$, and setting 
\[
\mathscr{L}_\pp(\F)={\rm tw}_{\boldsymbol{\xi}^{-1}}(\mathscr{L}_{\pp,\boldsymbol{\xi}}(\F)),
\] 
one then readily checks as in the proof of \cite[Thm.~2.11]{cas-2var} that for every $\phi\in\mathcal{X}_a^o(\cR)$, the measure $\phi(\mathscr{L}_\pp(\F))$ agrees with the measure constructed in \cite[\S{8.4}]{brooks} (in a formulation germane to that in \cite[\S{5.2}]{burungale-II}) attached to $\F_\phi$, from where the stated interpolation property follows from \cite[Prop.~8.9]{brooks}.
\end{proof}

\begin{cor}\label{cor:wan-bdp}
Let the hypotheses be as in Theorem~\ref{thm:bdp}, and denote by $\mathscr{L}_\pp(\F^\dagger/\cK)_{\ac}$ the image of  ${\rm tw}_{\Theta^{-1}}(\mathscr{L}_\pp(\F/\cK))$
under the natural projection $\cR^{\rm ur}[[\Gamma_\cK]]\rightarrow\cR^{\rm ur}[[\Gamma_\cK^{\ac}]]$. Then
\begin{equation}\label{eq:wan-bdp}
\mathscr{L}_\pp(\F^\dagger/\cK)_{\ac}=\mathscr{L}^{\tt BDP}_\pp(\F/\cK)^2
\end{equation}
up to a unit in $\cR^{\rm ur}[[\Gamma_\cK^\ac]][1/p]^\times$. In particular, $\mathscr{L}_\pp(\F^\dagger/\cK)_{\ac}$ is nonzero.
\end{cor}

\begin{proof}
In light of $(\ref{eq:factor-hida})$, the claimed equality up to a unit follows from a direct comparison of the interpolation properties in Theorem~\ref{thm:hida}, in $(\ref{eq:Katz})$, and in Theorem~\ref{thm:bdp} (\emph{cf.} \cite[\S{3.3}]{JSW}). On the other hand, by construction 
for every $\phi\in\mathcal{X}_a^o(\cR)$ the $p$-adic $L$-function $\mathscr{L}_\pp^{\tt BDP}(\F/\cK)$ specializes at $\phi$ to the $p$-adic $L$-functions 
constructed in \cite[\S{3.3}]{cas-hsieh1} (for $N^-=1$), and in \cite[\S{5.2}]{burungale-II} and \cite[\S{8}]{brooks} (for $N^->1$);  
since the latter are nonzero  by  \cite[Thm.~3.9]{cas-hsieh1} and \cite[Thm.~5.7]{burungale-II}, the last claim in the theorem follows.
\end{proof}


\section{Iwasawa theory}\label{sec:Iw}

Fix a prime $p>3$ and a positive integer $N$ prime to $p$, let 
\[
\F=\sum_{n=1}^\infty\boldsymbol{a}_nq^n\in\cR[[q]]
\]
be a Hida family 
of tame level $N$ and trivial tame character, and let $\cK$ be an imaginary quadratic field of discriminant prime of $Np$ in which $p=\pp\overline{\pp}$ splits. 

Let $\Sigma$ be a finite set of places of $\bQ$ containing $\infty$ and the primes dividing $Np$, and for any number field $F$, let $\mathfrak{G}_{F,\Sigma}$ denote
the Galois group of the maximal extension of $F$ unramified outside
the places above $\Sigma$.

%
%

\subsection{Selmer groups}\label{sec:selmer}

Let $T_\F$ be the big Galois representation associated to $\F$, 
for which we shall take the geometric realization denoted $M(\F)^*$ in \cite[Def.~7.2.5]{KLZ2}. In particular, $T_{\F}$ is a locally free $\cR$-module of rank $2$, and letting $D_p\subset G_{\bQ}$ be the decomposition group at $p$ determined by our fixed embedding $\iota_p:\overline{\bQ}\hookrightarrow\overline{\bQ}_p$, it fits in an exact sequence of $\cR[[D_p]]$-modules
\begin{equation}\label{eq:Gr-f}
0\rightarrow\mathscr{F}^+T_{\F}\rightarrow T_{\F}\rightarrow\mathscr{F}^-T_{\F}\rightarrow 0
\end{equation}
with $\mathscr{F}^\pm T_{\F}$ locally free of rank $1$ over $\cR$, and with the $D_p$-action on the quotient $\mathscr{F}^-T_{\F}$ given by the unramified character sending an arithmetic Frobenius to $\boldsymbol{a}_p\in\cR^\times$. 

Let $k_\cR:=\cR/\mathfrak{m}_\cR$ be the residue field of $\cR$, and denote by $\bar\rho_{\F}:G_\bQ\rightarrow{\rm GL}_2(\kappa_\cR)$ the semi-simple residual representation associated with $T_{\G}$,  which by (\ref{eq:Gr-f}) is conjugate to an upper-triangular representation upon restriction to $D_p$:
\[
\bar\rho_{\F}\vert_{D_p}\sim\left(\begin{smallmatrix}\bar{\varepsilon}&*\\ &\bar\delta\end{smallmatrix}\right).
\]
If $\bar{\rho}_{\F}$ is irreducible and $\bar\varepsilon\neq\bar\delta$, as we shall assume from now on, then by work of Wiles \cite{wiles88} (see also \cite[Thm.~7.2.8]{KLZ2}), $T_{\F}$ is free of rank $2$ over $\cR$, and each $\mathscr{F}^\pm T_{\F}$ is free of rank $1$. 

Recall that we let 
$\Gamma_\cK={\rm Gal}(\cK_\infty/\cK)$ denote the Galois group of the $\bZ_p^2$-extension of $\cK$, and
consider the $\cR[[\Gamma_\cK]]$-module
\[
\mathbf{T}_{}:=T_{\F}\otimes_{\cR}\cR[[\Gamma_\cK]]\quad
\]
equipped with the $G_\cK$-action via $\rho_{\F}\otimes\Psi_\cK$,  
where $\rho_{\F}$ 
is the $G_\bQ$-representation afforded by $T_\F$, and $\Psi_\cK$ is the tautological character $G_\cK\twoheadrightarrow\Gamma_\cK\hookrightarrow\cR[[\Gamma_\cK]]^\times$. Replacing $\Gamma_\cK$ by  $\Gamma_\cK^{\rm ac}$ (resp. $\Gamma_\cK^{\rm cyc}$), we define the $G_\cK$-module $\mathbf{T}^{\rm ac}$ (resp. $\mathbf{T}^{\rm cyc}$) similarly.

As in \cite{howard-invmath}, we also define the critical twist
\begin{equation}\label{eq:dagger}
T_{\F}^\dagger:=T_{\F}\otimes\Theta^{-1},
\end{equation}
where $\Theta:G_\bQ\rightarrow\cR^\times$ is the character (\ref{def:crit-char}), and define its deformations $\mathbf{T}^\dagger, \mathbf{T}^{\dagger,\ac}$, and $\mathbf{T}^{\dagger,{\rm cyc}}$ similarly as before. 

In the definitions that follow, we let $M$ denote either of the Galois modules just introduced,  
for which we naturally define $\mathscr{F}^\pm{M}$ using $(\ref{eq:Gr-f})$. 

Consider the \emph{$p$-relaxed} Selmer group defined by
\[
{\rm Sel}^{\{p\}}(F,M)={\rm ker}\Biggl\{\rH^1(\mathfrak{G}_{F,\Sigma},M)
\rightarrow\prod_{v\in\Sigma,\;v\nmid p}\frac{\rH^1(F_v,M)}{\rH^1_{\rm ur}(F_v,M)}\Biggr\},\nonumber
\]
where $\rH^1_{\rm ur}(F_v,M)={\rm ker}\{\rH^1(F_v,M)\rightarrow \rH^1(F_v^{\rm ur},M)\}$ is the unramified local condition.

\begin{defn}
For $v\vert p$ and $\mathscr{L}_v\in\{\emptyset,{\tt Gr},0\}$, set
\[
\rH^1_{\mathscr{L}_v}(F_v,M):=
\left\{
\begin{array}{ll}
\rH^1(F_v,M)&\textrm{if $\mathscr{L}_v=\emptyset$,}\\[0.1cm]
{\rm ker}\{\rH^1(F_v,M)\rightarrow\rH^1(F_v^{\rm ur},\mathscr{F}^-M)\}&\textrm{if $\mathscr{L}_v={\tt Gr}$,}\\[0.1cm]
\{0\}&\textrm{if $\mathscr{L}_v=0$,}
\end{array}
\right.
\]
and for $\mathscr{L}=\{\mathscr{L}_v\}_{v\vert p}$, define
\begin{equation}\label{def:auxsel}
{\rm Sel}_{\mathscr{L}}(F,M):={\rm ker}\Biggr\{{\rm Sel}^{\{p\}}(F,M)
\rightarrow\prod_{v\vert p}
\frac{\rH^1(F_{v},M)}{\rH_{\mathscr{L}_v}^1(F_{v},M)}\Biggr\}.\nonumber
\end{equation}
\end{defn}

Thus, for example ${\rm Sel}_{0,\emptyset}(\cK,M)$ is the subspace of ${\rm Sel}^{\{p\}}(\cK,M)$ consisting of classes which satisfy no condition (resp. are locally trivial) at $\overline{\pp}$ (resp. $\pp$). For the ease of notation, we let ${\rm Sel}_{\tt Gr}(F,M)$ denote the Selmer group ${\rm Sel}_{\mathscr{L}}(F,M)$ given by $\mathscr{L}_v={\tt Gr}$ for all $v\vert p$. 

We shall also need to consider Selmer groups for the discrete module
\[
A_{\F}:={\rm Hom}_{\rm cts}(T_{\F},\mu_{p^\infty}).
\] 
To define these, recall that by Shapiro's lemma there is a canonical isomorphism
\begin{equation}\label{eq:Shapiro}
\rH^1(\cK,\mathbf{T})\simeq\varprojlim_
{\cK\subset_{\rm f}F\subset\cK_\infty}\rH^1(F,T_\F),
\end{equation}
where $F$ runs over the finite extensions of $\cK$ contained in $\cK_\infty$ and the limits is with respect to the corestriction maps. By the compatibility of $(\ref{eq:Shapiro})$ with  the local restriction maps (see e.g. \cite[\S{3.1.2}]{SU}), the Selmer groups ${\rm Sel}_\mathscr{L}(\cK,\mathbf{T})$ are defined by local conditions ${\rm H}^1_{\mathscr{L}_v}(F_v,T_{\F})\subset{\rm H}^1(F_v,T_{\F})$ for all primes $v$. Thus we may let 
\[
{\rm Sel}_{\mathscr{L}}(\cK_\infty,A_\F)\subset\varinjlim_{\cK\subset_{\rm f} F\subset\cK_\infty}\rH^1(F,A_\F) 
\]
be the submodule cut out by 
the orthogonal complements of ${\rm H}^1_{\mathscr{L}_v}(F_v,T_{\F})$ under the perfect Tate duality
\[
\rH^1(F_v,T_\F)\times\rH^1(F_v,A_\F)\rightarrow
\bQ_p/\bZ_p.
\]
This also defines the Selmer groups ${\rm Sel}_\mathscr{L}(F,A_\F)\subset\rH^1(F,A_\F)$ for any number field $F$, and we shall also consider their variants for the twisted module 
\[
A_{\F}^\dagger:={\rm Hom}_{\bZ_p}(T_{\F}^\dagger,\mu_{p^\infty}),
\] 
or their specializations. 
Finally, if $W$ denotes any of the preceding discrete modules, we set
\[
X_{\mathscr{L}}(F,W):={\rm Hom}_{\bZ_p}({\rm Sel}_{\mathscr{L}}(F,W),\bQ_p/\bZ_p),
\]
which we simply denote by $X_{\tt Gr}(F,W)$ when $\mathscr{L}_v={\tt Gr}$ for all $v\vert p$.

We now record a number of lemmas for our later use.

\begin{lem}\label{lem:eq-ranks}
Assume that $\overline{\rho}_{\F}\vert_{G_F}$ is absolutely irreducible. Then ${\rm Sel}_{\tt Gr}(F,T_{\F}^\dagger)$ and $X_{\tt Gr}(F,A_{\F}^\dagger)$ have the same $\cR$-rank.
\end{lem}

\begin{proof}
For any height one prime $\mathfrak{P}\subset\cR$, let $\cR_{\mathfrak{P}}$ be the localization of $\cR$ at $\mathfrak{P}$, and let $F_\mathfrak{P}=\cR_{\mathfrak{P}}/\mathfrak{P}$ be the residue field.  It suffices to show that for all but finitely many $\mathfrak{P}\in\mathcal{X}_a(\cR)$, the spaces ${\rm Sel}_{\tt Gr}(F,T_{\F}^\dagger)_{\mathfrak{P}}/\mathfrak{P}$ and $X_{\tt Gr}(F,A_{\F}^\dagger)_{\mathfrak{P}}/\mathfrak{P}$ have the same $F_{\mathfrak{P}}$-dimension. 

As noted in \cite[\S{12.7.5}]{nekovar310} (see also \cite[Lem.~2.1.6]{howard-invmath}), Hida's results imply that  
the localization $\cR_\mathfrak{P}$ of $\cR$ at any $\mathfrak{P}\in\mathcal{X}_a(\cR)$ is a discrete valuation ring. Let $\pi\in\cR_{\mathfrak{P}}$ be a uniformizer. From Nekov{\'a}{\v{r}}'s theory (see \cite[Prop.~12.7.13.4(i)]{nekovar310}) and the identification \cite[(21)]{howard-invmath}, 
multiplication by $\pi$  
induces natural maps
\begin{align*}
{\rm Sel}_{\tt Gr}(F,T_{\F}^\dagger)_{\mathfrak{P}}/\pi&\hookrightarrow{\rm Sel}_{\tt Gr}(F,T_{\F,\mathfrak{P}}^\dagger/\pi),\\
{\rm Sel}_{\tt Gr}(F,A_{\F,\mathfrak{P}}^\dagger[\pi])&\twoheadrightarrow{\rm Sel}_{\tt Gr}(F,A_{\F}^\dagger)_{\mathfrak{P}}[\pi]
\end{align*}
which are isomorphisms for all but finitely many $\mathfrak{P}\in\mathcal{X}_a(\cR)$.   
Since by \cite[Lem.~1.3.3]{howard-PhD-I} the spaces ${\rm Sel}_{\tt Gr}(F,T_{\F,\mathfrak{P}}^\dagger/\pi)$ and ${\rm Sel}_{\tt Gr}(F,A_{\F,\mathfrak{P}}^\dagger[\pi])$ have the same $F_\mathfrak{P}$-dimension, the result follows.
\end{proof}

\begin{lem}\label{lem:no-tors}
If $\bar{\rho}_{\F}\vert_{G_\cK}$ is irreducible, then the modules
$\rH^1(\mathfrak{G}_{\cK,\Sigma},\bfT^\dagger)$ and $\rH^1(\mathfrak{G}_{\cK,\Sigma},\Tc)$ are torsion-free
over $\cR_{}[[\Gamma_\cK]]$ and $\cR_{}[[\Gamma^\ac_\cK]]$, respectively.
\end{lem}

\begin{proof}
This follows immediately from  \cite[\S{1.3.3}]{PR:Lp}, since  $\rH^0(\cK_\infty,\bar{\rho}_\F)=\rH^0(\cK_\infty^\ac,\bar{\rho}_\F)=\{0\}$ by the irreducibility of $\bar{\rho}_\F\vert_{G_\cK}$.
\end{proof}

\begin{lem}\label{lem:str-rel}
We have ${\rm rank}_{\cR[[\Gamma_\cK^\ac]]}(X_{{\tt Gr},\emptyset}(\cK_\infty^\ac,A_\F^\dagger))=1+{\rm rank}_{\cR[[\Gamma_\cK^\ac]]}(X_{{\tt Gr},0}(\cK_\infty^\ac,A_\F^{\dagger}))$. Moreover, if $\cR$ is regular then
\[
{\rm Char}_{\cR[[\Gamma_\cK^\ac]]}(X_{{\tt Gr},\emptyset}(\cK_\infty^\ac,A_\F^{\dagger})_{\rm tors})={\rm Char}_{\cR[[\Gamma_\cK^\ac]]}(X_{0,{\tt Gr}}(\cK_\infty^\ac,A_\F^{\dagger})_{\rm tors}),
\]
where the subscript {\rm tors} denotes the $\cR[[\Gamma_\cK^\ac]]$-torsion submodule.
\end{lem}

\begin{proof}
The first claim follows from an argument similar to that in Lemma~\ref{lem:eq-ranks} using part (2) of \cite[Lem.~2.3]{cas-BF}. For the second, note that the regularity of $\cR$ implies that of $\cR[[\Gamma_\cK^\ac]]$. Thus by \cite[Lem.~6.18]{Fouquet} the second claim follows from part (3) of \cite[Lem.~2.3]{cas-BF} (see also \cite[Prop.~3.16]{BL-ord}).
\end{proof}

We conclude this section with the following useful commutative algebra lemma from \cite{SU}, which will be used repeatedly in the proof of our main results in $\S\ref{sec:main}$.

\begin{lem}\label{lem:3.2}
Let $R$ be a local ring and $\mathfrak{a}\subset R$ a proper ideal such that $R/\mathfrak{a}$ is a domain. Let $I\subset R$ be an ideal and $\mathcal{L}$ an element of $R$ with $I\subset (\mathcal{L})$. Denote by a `bar' the image under the reduction map $R\rightarrow R/\mathfrak{a}$. If $\overline{\mathcal{L}}\in R/\mathfrak{a}$ is nonzero and $\overline{\mathcal{L}}\in\overline{I}$, then $I=(\mathcal{L})$.
\end{lem}

\begin{proof}
This is a special case of \cite[Lem.~3.2]{SU}.
\end{proof}

\subsection{Explicit reciprocity laws}

Let $G_\bQ$ act on $\Lambda_\Gamma$ via the tautological character $G_\bQ\twoheadrightarrow\Gamma\hookrightarrow\Lambda_\Gamma^\times$. In \cite{KLZ2}, Kings--Loeffler--Zerbes constructed special elements 
\[
_c\mathcal{BF}_m^{\F,\G}\in\rH^1(\bQ(\mu_m),T_{\F}\hat\otimes_{\bZ_p} T_{\G}\hat\otimes_{\bZ_p}\Lambda_\Gamma)
\] 
attached to pairs of Hida families $\F,\G$, and related the image of $_c\mathcal{BF}_1^{\F,\G}$ under a Perrin-Riou big logarithm map to the $p$-adic $L$-functions $L_p(\F,\G)$ and $L_p(\G,\F)$ of Theorem~\ref{thm:hida}. In this section we describe the variant of their results that we shall need.


Since $\bfT^\dagger=\bfT\otimes\Theta^{-1}$ by definition, the twist map ${\rm tw}_{\Theta^{-1}}:\cR[[\Gamma_\cK]]\rightarrow\cR[[\Gamma_\cK]]$ of $(\ref{eq:tw-theta})$ induces a $\cR$-linear isomorphism
\[
\widetilde{\rm tw}_{\Theta^{-1}}:\rH^1(\cK,\mathbf{T})\rightarrow\rH^1(\cK,\mathbf{T}^\dagger)
\]
satisfying $\widetilde{\rm tw}_{\Theta^{-1}}(\lambda x)={\rm tw}_{\Theta^{-1}}(\lambda)\widetilde{\rm tw}_{\Theta^{-1}}(x)$ for all $\lambda\in\cR[[\Gamma_\cK]]$.

\begin{thm}[Kings--Loeffler--Zerbes]
\label{thm:Col}
There exists a class $\mathcal{BF}_{}^\dagger\in{\rm Sel}_{{\tt Gr},\emptyset}(\cK,\mathbf{T}^\dagger)$ and $\cR[[\Gamma_\cK]]$-linear injections with pseudo-null cokernel
\begin{align*}
{\rm Col}^{(1),\dagger}:\rH^1(\cK_{\ppbar},\mathscr{F}^{-}\mathbf{T}^\dagger)
\;&\rightarrow\;I_\F\otimes_{\cR}\cR[[\Gamma_\cK]],\\
{\rm Col}^{(2),\dagger}:\rH^1(\cK_{\pp},\mathscr{F}^{+}\mathbf{T}^\dagger)
\;&\rightarrow\;I_\G\hat\otimes_{\bZ_p}\cR[[\Gamma_\cK]],
\end{align*}
where $\G$ is the CM Hida family in $(\ref{eq:g-CM})$, such that
\begin{align*}
{\rm Col}^{(1),\dagger}({\rm loc}_{\ppbar}(\mathcal{BF}^\dagger))&={\rm tw}_{\Theta^{-1}}(L_p^{}(\F,\G))\\
{\rm Col}^{(2),\dagger}({\rm loc}_\pp(\mathcal{BF}^\dagger))&={\rm tw}_{\Theta^{-1}}(L_p^{}(\G,\F)).
\end{align*}
In particular, for every prime $v$ of $\cK$ above $p$, the class ${\rm loc}_v(\mathcal{BF}^\dagger)\in\rH^1(\cK_v,\mathbf{T}^\dagger)$ is non-torsion over $\cR[[\Gamma_\cK]]$.
\end{thm}

\begin{proof}
This follows from the results of \cite{KLZ2}, as explained in  \cite[Thm.~2.4]{cas-BF}, and to which one needs to add some of the analysis in \cite{BST}.  
Indeed, taking $m=1$ in \cite[Def.~8.1.1]{KLZ2} (and using \cite[Lem.~6.8.9]{LLZ} to dispense with an auxiliary $c>1$ needed for the construction), one obtains a cohomology class 
\begin{equation}\label{eq:KLZ-class}
\mathcal{BF}^{\F,\G}\in\rH^1(\bQ,T_\F\hat\otimes_{\bZ_p}T_\G\hat\otimes_{\bZ_p}\Lambda_\Gamma) \nonumber
\end{equation}
attached to our fixed Hida family $\F$ and a second Hida family $\G$. 
Taking for $\G$ the canonical CM family in (\ref{eq:g-CM}), by \cite{BST} we have an isomorphism
\[
T_\G\simeq{\rm Ind}_\cK^\bQ\bZ_p[[\Gamma_\pp]]
\] 
as $G_\bQ$-modules, 
where the $G_\cK$-action on $\bZ_p[[\Gamma_\pp]]$ is given by the tautological character $G_\cK\twoheadrightarrow\Gamma_\pp\hookrightarrow\bZ_p[[\Gamma_\pp]]^\times$. By Shapiro's lemma, the class $\mathcal{BF}^{\F,\G}$ therefore defines a class $\mathcal{BF}\in\rH^1(\cK,\mathbf{T})$ whose image under $\widetilde{\rm tw}_{\Theta^{-1}}$ will be our required $\mathcal{BF}^\dagger$.

Indeed, the inclusion $\mathcal{BF}^\dagger\in {\rm Sel}_{{\tt Gr},\emptyset}(\cK,\bfT^\dagger)$ follows from \cite[Prop.~8.1.7]{KLZ2}, 
and by the explicit reciprocity law of \cite[Thm.~10.2.2]{KLZ2}, the maps
\[
{\rm Col}^{(1)}:=\langle\mathcal{L}(-),\eta_{\F}\otimes\omega_{\G}\rangle,\quad
{\rm Col}^{(2)}:=\left\langle\mathcal{L}(-),\eta_{\G}\otimes\omega_{\F}\right\rangle
\]
described in the proof of \cite[Thm.~2.4]{cas-BF} send
the restriction at $\overline{\pp}$ and $\pp$ of
$\mathcal{BF}$ to the $p$-adic $L$-functions $L_p(\F,\G)$ and $L_p^{}(\G,\F)$, respectively. Thus letting ${\rm Col}^{(1),\dagger}$ and ${\rm Col}^{(2),\dagger}$ be the $\cR[[\Gamma_\cK]]$-linear maps defined by the commutative diagrams
\[
\xymatrix{
\rH^1(\cK_{\ppbar},\mathscr{F}^{-}\mathbf{T})\ar[r]^-{{\rm Col}^{(1)}}\ar[d]^-{\widetilde{\rm tw}_{\Theta^{-1}}}&I_\F\otimes_{\cR}\cR[[\Gamma_\cK]]\ar[d]^-{{\rm tw}_{\Theta^{-1}}}&
\rH^1(\cK_{\pp},\mathscr{F}^{+}\mathbf{T})\ar[r]^-{{\rm Col}^{(2)}}\ar[d]^-{\widetilde{\rm tw}_{\Theta^{-1}}}&I_\G\hat\otimes_{\bZ_p}\cR[[\Gamma_\cK]]\ar[d]^-{{\rm tw}_{\Theta^{-1}}}\\
\rH^1(\cK_{\ppbar},\mathscr{F}^{-}\mathbf{T}^\dagger)\ar[r]^-{{\rm Col}^{(1),\dagger}}&I_\F\otimes_{\cR}\cR[[\Gamma_\cK]]&
\rH^1(\cK_{\pp},\mathscr{F}^{+}\mathbf{T}^\dagger)\ar[r]^-{{\rm Col}^{(2),\dagger}}&I_\G\hat\otimes_{\bZ_p}\cR[[\Gamma_\cK]],
}
\]
the result follows, with the last claim being an immediate consequence of the nonvanishing of the $p$-adic $L$-functions $L_p(\F,\G)$ and $L_p(\G,\F)$ (see e.g.  \cite[Rem.~1.3]{cas-BF}).
\end{proof}	
		
We shall also need to consider anticyclotomic variants  of the maps ${\rm Col}^{(i),\dagger}$ in Theorem~\ref{thm:Col}. Letting $\mathcal{I}_{\rm cyc}$ be the kernel of the natural projection $\cR[[\Gamma_\cK]]\rightarrow\cR[[\Gamma_{\cK}^\ac]]$, the map
\[
{\rm Col}_\ac^{(1),\dagger}:\rH^1(\cK_{\ppbar},\mathscr{F}^-\bfT^{\dagger,\ac})\rightarrow I_\F\otimes_\cR\cR[[\Gamma_\cK^\ac]]
\]
is defined by reducing ${\rm Col}^{(1),\dagger}$ modulo the ideal $\mathcal{I}_{\rm cyc}$, using the fact that by the vanishing of $\rH^0(\cK_{\ppbar},\mathscr{F}^-\bfT^{\dagger,\ac})$ the restriction map induces a natural isomorphism
\[
\rH^1(\cK_{\ppbar},\mathscr{F}^-\bfT^{\dagger})/\mathcal{I}_{\rm cyc}\simeq\rH^1(\cK_{\ppbar},\mathscr{F}^-\bfT^{\dagger,\ac}).
\] 
The map ${\rm Col}_\ac^{(2),\dagger}:\rH^1(\cK_{\pp},\mathscr{F}^+\bfT^{\dagger,\ac})\rightarrow I_\G\hat\otimes_{\bZ_p}\cR[[\Gamma_\cK^\ac]]$ is defined in the same manner. 

Note that since the maps ${\rm Col}^{(i),\dagger}$ are injective with pseudo-null cokernel, the same is true for the maps ${\rm Col}_\ac^{(i),\dagger}$.

\begin{cor}\label{cor:str-Gr}
Let $\mathcal{BF}^{\dagger,\ac}$ be the image of the class $\mathcal{BF}^\dagger$ under the natural map $\rH^1(\cK,\mathbf{T}^\dagger)\rightarrow\rH^1(\cK,\mathbf{T}^{\dagger,\ac})$. 
Assume that $\cK$ satisfies the hypothesis {\rm (\ref{eq:gen-Heeg-f})}. Then we have the inclusion
\[
{\rm loc}_{\overline{\pp}}(\mathcal{BF}_{}^{\dagger,\ac})\in{\rm ker}\{\rH^1(\cK_{\overline\pp},\Tc)\rightarrow\rH^1(\cK_{\overline{\pp}},\mathscr{F}^-\Tc)\};
\] 
in particular, $\mathcal{BF}_{}^{\dagger,\ac}\in{\rm Sel}_{\tt Gr}(\cK,\Tc)$. Moreover, if we assume in addition that $N$ is squarefree when $N^->1$, then ${\rm loc}_{\pp}(\mathcal{BF}_{}^{\dagger,\ac})$ is non-torsion over $\cR[[\Gamma_\cK^\ac]]$.
\end{cor}

\begin{proof}
The combination of Theorem~\ref{thm:Col} and Proposition~\ref{thm:hida-1} yields the vanishing of the image of ${\rm loc}_{\ppbar}(\mathcal{BF}^{\dagger,\ac})$ under the map ${\rm Col}_\ac^{(1),\dagger}$, so the first claim follows from its injectivity. The second claim follows from Theorem~\ref{thm:Col} together with the novanishing result of Corollary~\ref{cor:wan-bdp}.
\end{proof}

\subsection{Iwasawa main conjectures}\label{sec:ES}

We now use the reciprocity laws of Theorem~\ref{thm:Col} to relate different variants of the Iwasawa main conjecture. 



\begin{thm}\label{thm:2-varIMC}
Assume that $\overline{\rho}_{\F}\vert_{G_\cK}$ is irreducible. 
Then the following are equivalent:
\begin{enumerate}
\item[(i)]{} $X_{{\tt Gr},0}(\cK_\infty,A_\F^{\dagger})$ is $\cR[[\Gamma_\cK]]$-torsion, ${\rm Sel}_{{\tt Gr},\emptyset}(\cK,\mathbf{T}^{\dagger})$ has $\cR[[\Gamma_\cK]]$-rank one, 
and
\begin{equation}\label{eq:BF-IMC}
{\rm Char}_{\cR[[\Gamma_\cK]]}(X_{{\tt Gr},0}(\cK_\infty,A_\F^{\dagger}))=
{\rm Char}_{\cR[[\Gamma_\cK]]}\bigg(\frac{{\rm Sel}_{{\tt Gr},\emptyset}(\cK,\mathbf{T}^{\dagger})}{\cR[[\Gamma_\cK]]\cdot\mathcal{BF}^{\dagger}}\biggr)
\nonumber
\end{equation}
up to powers of $p$. 
\item[(ii)]{} Both $X_{\emptyset,0}(\cK_\infty,A_\F^{\dagger})$ and ${\rm Sel}_{0,\emptyset}(\cK,\mathbf{T}^{\dagger})$ are $\cR[[\Gamma_\cK]]$-torsion, and
\[
{\rm Char}_{\cR[[\Gamma_\cK]]}(X_{\emptyset,0}(\cK_\infty,A_\F^{\dagger}))\cdot\cR^{\rm ur}[[\Gamma_\cK]]=({\rm tw}_{\Theta^{-1}}(\mathscr{L}_\pp(\F/\cK)))
\]
up to powers of $p$. 
\item[(iii)]{} Both $X_{\tt Gr}(\cK_\infty,A_\F^{\dagger})$ and ${\rm Sel}_{\tt Gr}(\cK,\mathbf{T}^{\dagger})$ are $\cR[[\Gamma_\cK]]$-torsion, and
\[
{\rm Char}_{\cR[[\Gamma_\cK]]}(X_{\tt Gr}(\cK_\infty,A_\F^{\dagger}))=({\rm tw}_{\Theta^{-1}}(L_p^{\tt Hi}(\F/\cK))).
\]
up to powers of $p$.
\end{enumerate}
Moreover, if in addition $\cK$ satisfies the hypothesis {\rm (\ref{eq:gen-Heeg-f})}, with $N$ being squarefree when $N^->1$, then the following are equivalent:
\begin{enumerate}
\item[(i)']{} $X_{{\tt Gr},0}(\cK_\infty^\ac,A_\F^{\dagger})$ is $\cR[[\Gamma_\cK^\ac]]$-torsion, ${\rm Sel}_{{\tt Gr},\emptyset}(\cK,\mathbf{T}^{\dagger,\ac})$ has $\cR[[\Gamma_\cK^\ac]]$-rank one, 
and
\begin{equation}\label{eq:BF-IMC}
{\rm Char}_{\cR[[\Gamma_\cK^\ac]]}(X_{{\tt Gr},0}(\cK_\infty^\ac,A_\F^{\dagger}))=
{\rm Char}_{\cR[[\Gamma_\cK^\ac]]}\bigg(\frac{{\rm Sel}_{{\tt Gr},\emptyset}(\cK,\mathbf{T}^{\dagger,\ac})}{\cR[[\Gamma_\cK^\ac]]\cdot\mathcal{BF}^{\dagger,\ac}}\biggr)
\nonumber
\end{equation}
up to powers of $p$.
\item[(ii)']{} Both $X_{\emptyset,0}(\cK_\infty^\ac,A_\F^{\dagger})$ and ${\rm Sel}_{0,\emptyset}(\cK,\mathbf{T}^{\dagger,\ac})$ are $\cR[[\Gamma_\cK^\ac]]$-torsion, and
\[
{\rm Char}_{\cR[[\Gamma_\cK^\ac]]}(X_{\emptyset,0}(\cK_\infty^\ac,A_\F^{\dagger}))\cdot\cR^{\rm ur}[[\Gamma_\cK^\ac]]=(\mathscr{L}^{\tt BDP}_\pp(\F/\cK)^2)
\]
up to powers of $p$. 
\end{enumerate}
\end{thm}

\begin{proof}
Consider the exact sequence coming form Poitou--Tate duality
\begin{align*}\label{eq:ES-1b}
0\rightarrow{\rm Sel}_{0,\emptyset}(\cK,\bfT^\dagger)\rightarrow{\rm Sel}_{{\tt Gr},\emptyset}(\cK,\bfT^\dagger)\xrightarrow{{\rm loc}_\pp}
&\rH^1_{\tt Gr}(\cK_{\pp},\bfT^\dagger)\\
&\rightarrow
X_{\emptyset,0}(\cK_\infty,A_\F^\dagger)\rightarrow X_{{\tt Gr},0}(\cK_\infty,A_\F^\dagger)\rightarrow 0.\nonumber
\end{align*}
By Theorem~\ref{thm:Col}, the cokernel of the map ${\rm loc}_\pp$ is $\cR[[\Gamma_\cK]]$-torsion, and so the equivalence between the claimed ranks in (i) and (ii) follows. By Lemma~\ref{lem:no-tors}, if  ${\rm Sel}_{0,\emptyset}(\cK,\bfT^\dagger)$ is $\cR[[\Gamma_\cK]]$-torsion then it is trivial, and so the above yields
\begin{equation}\label{eq:div-1}
0\rightarrow\frac{{\rm Sel}_{{\tt Gr},\emptyset}(\cK,\bfT^\dagger)}{\cR[[\Gamma_\cK]]\cdot\mathcal{BF}^\dagger}\xrightarrow{{\rm loc}_\pp}
\frac{\rH^1_{\tt Gr}(\cK_{\pp},\bfT^\dagger)}{\cR[[\Gamma_\cK]]\cdot{\rm loc}_\pp(\mathcal{BF}^\dagger)}\\
\rightarrow
X_{\emptyset,0}(\cK_\infty,A_\F^\dagger)\rightarrow X_{{\tt Gr},0}(\cK_\infty,A_\F^\dagger)\rightarrow 0.
\end{equation}
By \cite[Cor.~4.13]{betina-dimitrov-CM}, the congruence ideal of the  CM Hida family $\G$ in $(\ref{eq:g-CM})$ is generated by $\mathscr{L}_\pp(\cK)_\ac$ after inverting $p$, and therefore by Theorem~\ref{thm:Col} and (\ref{eq:factor-hida}) the map ${\rm Col}_{}^{(2),\dagger}$ multiplied by this generator yields an injection
\begin{equation}\label{eq:katz-col}
\frac{\rH^1_{\tt Gr}(\cK_{\pp},\bfT^\dagger)\cdot\cR^{\rm ur}[[\Gamma_\cK]][1/p]}{\cR^{\rm ur}[[\Gamma_\cK]][1/p]\cdot{\rm loc}_\pp(\mathcal{BF}^\dagger))}\hookrightarrow\frac{\cR^{\rm ur}[[\Gamma_\cK]][1/p]}{({\rm tw}_{\Theta^{-1}}(\mathscr{L}_\pp(\F/\cK)))}\nonumber
\end{equation}
with pseudo-null cokernel, which combined with $(\ref{eq:div-1})$ completes the proof of the equivalence of (i) and (ii). The equivalence between (i)' and (ii)' when $\cK$ satisfies the hypothesis (\ref{ass:gen-H}) is shown in the same way, using the nonvanishing of ${\rm loc}_\pp(\mathcal{BF}^{\dagger,\ac})$ from Corollary~\ref{cor:str-Gr}. 

Now consider the exact sequence
\begin{align*}\label{eq:ES-2b}
0\rightarrow{\rm Sel}_{\tt Gr}(\cK,\bfT^\dagger)\rightarrow{\rm Sel}_{{\tt Gr},\emptyset}(\cK,\bfT^\dagger)\xrightarrow{{\rm loc}_{\overline\pp}}
&\frac{\rH^1(\cK_{\overline\pp},\bfT^\dagger)}{\rH^1_{\tt Gr}(\cK_{\overline\pp},\bfT^\dagger)}\simeq\rH^1(\cK_{\overline{\pp}},\mathscr{F}^-\bfT^\dagger)\\
&\rightarrow
X_{\tt Gr}(\cK_\infty,A_\F^\dagger)\rightarrow X_{{\tt Gr},0}(\cK_\infty,A_\F^\dagger)\rightarrow 0,\nonumber
\end{align*}
which similarly as before implies the equivalence between the claimed $\cR[[\Gamma_\cK]]$-ranks in (ii) and (iii), and by Theorem~\ref{thm:Col} and Lemma~\ref{lem:no-tors} yields the exact sequence
\[
0\rightarrow\frac{
{\rm Sel}_{{\tt Gr},\emptyset}(\cK,\bfT^\dagger)}{\cR[[\Gamma_\cK]]\cdot\mathcal{BF}^\dagger}\xrightarrow{{\rm loc}_{\overline\pp}}
\frac{\rH^1(\cK_{\overline{\pp}},\mathscr{F}^-\bfT^\dagger)}{\cR[[\Gamma_\cK]]\cdot{\rm loc}_{\pp}(\mathcal{BF}^\dagger)}\rightarrow
X_{\tt Gr}(\cK_\infty,A_\F^\dagger)\rightarrow X_{{\tt Gr},0}(\cK_\infty,A_\F^\dagger)\rightarrow 0.\nonumber
\]
Since by Theorem~\ref{thm:Col} the map ${\rm Col}_{}^{(1),\dagger}$  multiplied by a generator of the congruence ideal of $\F$ yields an injection $\rH^1(\cK_{\overline{\pp}},\mathscr{F}^-\bfT^\dagger)\rightarrow\cR[[\Gamma_\cK]]$ with pseudo-null cokernel sending ${\rm loc}_{\pp}(\mathcal{BF}^\dagger)$ into ${\rm tw}_{\Theta^{-1}}(L_p^{\tt Hi}(\F/\cK))$ up to a unit in $\cR^\times$, 
the equivalence between (ii) and (iii) follows. 
\end{proof}

\subsection{Rubin's height formula}\label{sec:main}


Fix a topological generator $\gamma_{\rm cyc}\in\Gamma_\cK^{\rm cyc}$, and using the identification $\cR[[\Gamma_\cK]]\simeq(\cR[[\Gamma_\cK^\ac]])[[\Gamma_\cK^{\rm cyc}]]$, expand 
\begin{equation}\label{eq:2-exp}
{\rm tw}_{\Theta^{-1}}(L_p^{\tt Hi}(\F/\cK))=L_{p,0}^{\tt Hi}(\F^\dagger/\cK)_{\ac}+L_{p,1}^{\tt Hi}(\F^\dagger/\cK)_{\ac}\cdot(\gamma_{\rm cyc}-1)+\cdots
\end{equation}
as a power series in $\gamma_{\rm cyc}-1$. The constant term $L_{p,0}^{\tt Hi}(\F^\dagger/\cK)_{\ac}$ thus corresponds to the image of ${\rm tw}_{\Theta^{-1}}(L_p^{\tt Hi}(\F/\cK))$ under the natural projection $\cR[[\Gamma_\cK]]\rightarrow\cR[[\Gamma_\cK^\ac]]$.


By Shapiro's lemma, 
we may consider the class $\mathcal{BF}^\dagger\in{\rm Sel}_{{\tt Gr},\emptyset}(\cK,\bfT^\dagger)$ of Theorem~\ref{thm:Col} as a system of classes  
$\mathcal{BF}_F^\dagger\in{\rm Sel}_{{\tt Gr},\emptyset}(F,T_{\F}^\dagger)$, 
indexed by the finite extensions $F$ of $\cK$ contained in $\cK_\infty$, compatible under the corestriction maps. For any intermediate extension $\cK\subset L\subset \cK_\infty$, we then set 
\[
\mathcal{BF}^\dagger_{}(L):=\{\mathcal{BF}^\dagger_{F}\}_{\cK\subset_{\rm f} F\subset L}
\] 
with $F$ running over the finite extensions of $\cK$ contained in $L$, 
so in particular $\mathcal{BF}^\dagger_{}(\cK_\infty)$ is nothing but $\mathcal{BF}^\dagger_{}$. 

Let $\cK_n^\ac$ be the subextension of $\cK_\infty^\ac$ with $[\cK_n^\ac:\cK]=p^n$, define $\cK_k^{\rm cyc}$ similarly, and set $L_{n,k}=\cK_n^\ac\cK_k^{\rm cyc}$ for all $k\leqslant\infty$.

\begin{lem}\label{lem:3.1.1}
Assume that $\cK$ satisfies the hypothesis {\rm (\ref{eq:gen-Heeg-f})} and that $\bar{\rho}_\F\vert_{G_\cK}$ is irreducible. Then there is a unique element
\[
\beta^\dagger_{n}\in \rH^1_{}(\cK^\ac_{n,\overline\pp},\mathscr{F}^-\mathbf{T}^{\dagger,{\cyc}})
\]
such that ${\rm loc}_{\overline{\pp}}(\mathcal{BF}^\dagger_{}(L_{n,\infty}))=(\gamma_{\cyc}-1)\beta_n^\dagger$. 
Furthermore, the natural images of the classes $\beta_n^\dagger$ in $\rH^1(\cK_{n,\overline\pp}^\ac,\mathscr{F}^-T_{\F}^\dagger)$ are norm-compatible, defining a class 
\[
\{\beta_n^\dagger(\mathds{1})\}_n\in\varprojlim_n\rH^1(\cK_{n,\overline\pp}^{\ac},\mathscr{F}^-T_{\F}^\dagger)\simeq\rH^1(\cK_{\overline\pp},\mathscr{F}^-\Tc)
\]
that is sent to the linear term $L_{p,1}^{\tt Hi}(\F^\dagger/\cK)_{\ac}$ under the 
map ${\rm Col}_{\ac}^{(1),\dagger}$. 
\end{lem}

\begin{proof}
By the explicit reciprocity law of Theorem~\ref{thm:Col}, the first claim follows from the vanishing of $L_{p,0}^{\tt Hi}(\F^\dagger/\cK)_{\ac}$ (see Proposition~\ref{thm:hida-1}) and the injectivity of ${\rm Col}^{(1),\dagger}$, with the uniqueness claim being an immediate consequence of Lemma~\ref{lem:no-tors}; the last claim is a direct consequence of the definitions of $\beta^\dagger_n$ and $L^{\tt Hi}_{p,1}(\F^\dagger/\cK)_{\ac}$.
\end{proof}

Let $\cI^{\rm cyc}=(\gamma_{\rm cyc}-1)\subset\cR[[\Gamma_\cK^\cyc]]$ be the augmentation ideal,
and set $\mathcal{J}^{\rm cyc}=\cI^{\rm cyc}/(\cI^{\rm cyc})^2$. By work of Plater \cite{Plater}, and more generally Nekov{\'a}{\v{r}} \cite[\S{11}]{nekovar310}, for every $n$ there is a canonical (up to sign) $\cR$-adic height pairing
\begin{equation}\label{eq:I-ht}
\langle,\rangle_{\cK_n^\ac,\cR}^{\cyc}:{\rm Sel}_{\tt Gr}(\cK_n^\ac,T_\F^\dagger)\times{\rm Sel}_{\tt Gr}(\cK_n^\ac,T_\F^\dagger)
\rightarrow\mathcal{J}^{\rm cyc}\otimes_\cR F_\cR.
\end{equation}

(Note that the local indecomposability hypothesis (H1) in \cite[p.~107]{Plater} is only used to ensure the existence of well-defined sub and quotients at the places above $p$, which for $T_\F^\dagger$ is automatic, while hypotheses (H2) and (H3) in \emph{loc.cit.} follow from \cite[Lem.~2.4.4]{howard-invmath} for $T_\F^\dagger$.)

Denoting by $\rH^1_{\tt Gr}(L_{n,k,v},T_\F^\dagger)\subset\rH^1(L_{n,k,v},T_\F^\dagger)$ the local condition defining ${\rm Sel}_{\tt Gr}(L_{n,k},T_\F^\dagger)$ at a place $v$, 
Plater's definition of $\langle,\rangle_{\cK_n^\ac,\cR}^{\cyc}$ (which we shall briefly recall in the proof of Proposition~\ref{thm:rubin-ht} below)   
shows that $\langle,\rangle_{\cK_n^\ac,\cR}^{\cyc}$ takes integral values in the submodule of ${\rm Sel}_{\tt Gr}(\cK_n^\ac,T_\F^\dagger)$ consisting of classes which are local cyclotomic universal norms at all places $v$ above $p$, i.e., classes in
\[
\rH^1_{\tt Gr}(\cK_{n,v}^\ac,T_\F^\dagger)^{\rm univ}:=\bigcap_{k}{\rm cor}_{L_{n,k,v}/\cK_{n,v}^\ac}(\rH^1_{\tt Gr}(L_{n,k,v},T_\F^\dagger)),
\]
and so by 
\cite[Lem.~2.3.1]{PR-109} the denominators of (\ref{eq:I-ht}) are abounded independently of $n$. 

The next result generalizes the height formula of \cite[Thm.~3.2(ii)]{rubin-ht} to our context. 

\begin{prop}\label{thm:rubin-ht}
Assume that $\cK$ satisfies the hypothesis {\rm (\ref{ass:gen-H})}  and that $\bar{\rho}_\F\vert_{G_\cK}$ is irreducible. Then the classes $\mathcal{BF}^\dagger_{\cK_n^\ac}$ land in ${\rm Sel}_{\tt Gr}(\cK_n^\ac,T_\F^\dagger)$, and for every $x\in{\rm Sel}_{\tt Gr}(\cK_n^\ac,T_\F^\dagger)$ we have
\begin{equation}\label{eq:rubin-ht}
\langle\mathcal{BF}^\dagger_{\cK_n^\ac},x\rangle^{\rm cyc}_{\cK_n^\ac,\cR}=(\beta_n^\dagger(\mathds{1}),{\rm loc}_{\overline\pp}(x))_{\cK_{n,\ppbar}^\ac}\otimes(\gamma_{\rm cyc}-1),
\end{equation}
where $(,)_{\cK_{n,\ppbar}^\ac}$ is the local Tate pairing
\[
\frac{\rH^1(\cK_{n,\overline\pp}^\ac,T_\F^\dagger)}{\rH^1_{\tt  Gr}(\cK_{n,\overline\pp}^\ac,T_\F^\dagger)}
\times \rH^1_{\tt Gr}(\cK_{n,\overline\pp}^\ac,T_\F^\dagger)\rightarrow\cR.
\]
\end{prop}

\begin{proof}
The first claim follows from the explicit reciprocity law of Theorem~\ref{thm:Col}, the vanishing of $L_{p,0}^{\tt Hi}(\F^\dagger/\cK)_{\ac}$, and the injectivity of ${\rm Col}^{(1),\dagger}$. On the other hand, the proof of formula $(\ref{eq:rubin-ht})$ could be deduced from the general result \cite[(11.3.14)]{nekovar310}, but shall give a proof following the more direct generalization of Rubin's formula contained in \cite[\S{3}]{arnold-ht}.

We begin by recalling Plater's definition of the $\cR$-adic height pairing (itself a generalization of Perrin-Riou's \cite[\S{1.2}]{PR-109} in the $p$-adic setting). Let $\lambda$ be the isomorphism $\Gamma_\cK^{\rm cyc}\simeq\mathcal{J}^{\rm cyc}$ sending $\gamma_{\rm cyc}$ to the class of $\gamma_{\rm cyc}-1$. Composing with the natural isomorphism ${\rm Gal}(L_{n,\infty}/\cK_n^\ac)\simeq\Gamma^{\rm cyc}$ the map $\lambda$ defines a class in $\rH^1(\cK_n^\ac,\mathcal{J}^{\rm cyc})$, where we equip $\mathcal{J}^{\rm cyc}$ with the trivial Galois action, and so taking cup product  we get
\[
\rho_v:\rH^1(K_{n,v}^\ac,\cR(1))\xrightarrow{\cup{\rm loc}_v(\lambda)}\rH^2(K_{n,v}^\ac,\mathcal{J}^{\rm cyc}(1))\simeq\mathcal{J}^{\rm cyc}
\]
for every place $v$.


Denote by ${\rm Sel}_{\tt Gr}(\cK_n^\ac,T_\F^\dagger)^{\rm univ}$ the  submodule of ${\rm Sel}_{\tt Gr}(\cK_n^\ac,T_\F^\dagger)$ (with $\cR$-torsion quotient, as noted earlier) consisting of classes lying in $\rH^1_{\tt Gr}(\cK_{n,v}^\ac,T_\F^\dagger)^{\rm univ}$ for all $v\mid p$, and let $x, y\in{\rm Sel}_{\tt Gr}(\cK_n^\ac,T_\F^\dagger)^{\rm univ}$. Then $x$ corresponds to an extension of Galois modules
\[
0\rightarrow T_\F^\dagger\rightarrow X\rightarrow\cR\rightarrow 0.
\]
The Kummer dual of this sequence induces maps on cohomology
\[
\rH^1(\cK_{n}^\ac,X^*(1))\rightarrow\rH^1(\cK_{n}^\ac,T_\F^\dagger)\xrightarrow{\delta}\rH^2(\cK_{n}^\ac,\cR(1))
\]
such that $\delta(y)=0$ (since $\rH^2(\cK_{n}^\ac,\cR(1))$ injects into $\bigoplus_v\rH^2(\cK_{n,v}^\ac,\cR(1))$ and the $v$-th component of $\delta(y)$ is given by ${\rm loc}_v(y)\cup{\rm loc}_v(x)=0$ by the self-duality of Greenberg's local conditions). Thus   $y$ is the image of some $y^{\rm glob}\in\rH^1(\cK_n^\ac,X^*(1))$. 

On the other hand, if $v$ is any place of $\cK_{n}^\ac$, for every $k$ we can write ${\rm loc}_v(y)={\rm cor}_{L_{n,k,v}/\cK_{n,v}^\ac}(y_{k,v})$ for some $y_{k,v}\in\rH^1_{\tt Gr}(L_{n,k,v},T_\F^\dagger)$, and by a similar argument as above there exists a class $\widetilde{y}_{k,v}\in\rH^1(L_{n,k,v},X^*(1))$ lifting $y_{k,v}$ under the natural map $\pi_v$ in the exact sequence
\begin{equation}\label{eq:local-v}
\rH^1(L_{n,k,v},X^*(1))\xrightarrow{\pi_v}\rH^1(L_{n,k,v},T_\F^\dagger)\xrightarrow{\delta_v}\rH^2(L_{n,k,v},\cR(1)).
\end{equation}
The difference ${\rm loc}_v(y^{\rm glob})-{\rm cor}_{L_{n,k,v}/\cK_{n,v}^\ac}(\widetilde{y}_{k,v})$ is then the image of some class $w_{k,v}\in\rH^1(\cK_{n,v}^\ac,\cR(1))$, and we define
\[
\langle y,x\rangle_{\cK_n^\ac,\cR}^{\rm cyc}:=\lim_{k\to\infty}\sum_v\rho_v(w_{k,v}),
\]
a limit which is easily checked to exist and be independent of all choices. If in addition $y=y_0$ is the base class of a compatible system of classes 
\[
y_\infty=\{y_k\}_k\in\rH^1(\cK_n^\ac,\mathbf{T}^{\dagger,{\rm cyc}})=\varprojlim_k\rH^1(L_{n,k},T_\F^\dagger),
\] 
then one easily checks (see e.g. \cite[Lem.~3.2.2]{AHsplit}) that there are classes $y_k^{\rm glob}\in\rH^1(L_{n,k},X^*(1))$ lifting $y_k$. Similarly as above, for every place $v$ of $L_{n,k}$ the corestriction of ${\rm loc}_v(y_k^{\rm glob})-\widetilde{y}_{k,v}$ to $\rH^1(K_{n,v}^\ac,X^*(1))$ is the image of a class $w'_{k,v}\in\rH^1(K_{n,v}^\ac,\cR(1))$, and with these choices we see that the above expression for $\langle y,x\rangle_{\cK_n^\ac,\cR}^{\rm cyc}$ reduces to
\begin{equation}\label{eq:ht-p}
\langle y,x\rangle_{\cK_n^\ac,\cR}^{\rm cyc}=\lim_{k\to\infty}\sum_{v\mid p}\rho_v(w'_{k,v}).
\end{equation}

As in \cite[\S{3.8}]{arnold-ht}, division by $\gamma_{\rm cyc}-1$ defines a natural \emph{derivative map}
\[
\mathfrak{Der}_{}:\rH^1(\cK_{n,v}^\ac,T_\F^\dagger\otimes_\cR\mathcal{I}^{\rm cyc})\rightarrow\rH^1(\cK_{n,v}^\ac,T_\F^\dagger)
\] 
whose composition with the natural projection $\rH^1(\cK_{n,v}^\ac,T_\F^\dagger)\rightarrow\rH^1(\cK_{n,v}^\ac,\fil^-T_\F^\dagger)$ factors as
\begin{equation}\label{eq:der-v}
\begin{aligned}
\xymatrix{
\rH^1(\cK_{n,v}^\ac,T_\F^\dagger\otimes_\cR\mathcal{I}^{\rm cyc})\ar@{->>}[r]\ar[d]_-{\mathfrak{Der}_{}}&
\rH^1(\cK_{n,v}^\ac,\fil^-T_\F^\dagger\otimes_\cR\mathcal{I}^{\rm cyc})\ar[d]^-{\mathfrak{Der}_{-}}\\
\rH^1(\cK_{n,v}^\ac,T_\F^\dagger)\ar@{->>}[r]&\rH^1(\cK_{n,v}^\ac,\fil^-T_\F^\dagger).
}
\end{aligned}
\end{equation}

Letting ${\rm pr}_{\mathds{1}}$ be the natural projection $\rH^1(\cK_{n,v}^\ac,X^*(1)\otimes_\cR\cR[[\Gamma^{\rm cyc}]])\rightarrow\rH^1(\cK_{n,v}^\ac,X^*(1))$, the  expression $(\ref{eq:ht-p})$ for $\langle y,x\rangle_{\cK_n^\ac,\cR}^{\rm cyc}$ can be rewritten as 
\[
\langle y,x\rangle_{\cK_n^\ac,\cR}^{\rm cyc}=\sum_{v\vert p}{\rm pr}_{\mathds{1}}({\rm loc}_v(y_\infty^{\rm glob})-\widetilde{y}_{\infty,v}),
\]
where ${\rm loc}_v(y_\infty^{\rm glob})-\widetilde{y}_{\infty,v}\in\rH^1(\cK_{n,v}^\ac,X^*(1)\otimes_\cR\cR[[\Gamma^{\rm cyc}]])$ is a lift of ${\rm loc}_v(y_\infty)-y_{\infty,v}\in\rH^1(\cK_{n,v}^\ac,T_\F^\dagger\otimes\mathcal{I}^{\rm cyc})$, and hence by \cite[Prop.~3.10]{arnold-ht} we obtain
\begin{equation}\label{eq:der}
\begin{aligned}
\langle y,x\rangle_{\cK_n^\ac,\cR}^{\rm cyc}&=\sum_{v\mid p}\delta_v\left(\mathfrak{Der}({\rm loc}_v(y_\infty)-y_{\infty,v})\right)\otimes(\gamma^{\rm cyc}-1)\\
&=\sum_{v\mid p}\left(\mathfrak{Der}({\rm loc}_v(y_\infty)-y_{\infty,v}),{\rm loc}_v(x)\right)_{\cK_{n,v}^\ac}\otimes(\gamma^{\rm cyc}-1)\\
&=\sum_{v\mid p}\left(\mathfrak{Der}_{-}({\rm loc}_v(y_\infty)),{\rm loc}_v(x)\right)_{\cK_{n,v}^\ac}\otimes(\gamma^{\rm cyc}-1),
\end{aligned}
\end{equation}
where the last equality follows from the commutativity of $(\ref{eq:der-v})$ and the fact that $y_{\infty,v}=\{y_{k,v}\}_k$ has trivial image in $\rH^1(\cK_{n,v}^\ac,\fil^-\mathbf{T}^{\dagger,{\rm cyc}})$. 

Now taking $y_\infty=\mathcal{BF}^\dagger(L_{n,\infty})$ in $(\ref{eq:der})$ we see that the contribution to $\langle\mathcal{BF}^\dagger_{\cK_n^\ac},x\rangle^{\rm cyc}_{\cK_n^\ac,\cR}$ from $\pp$ is zero, since $\mathcal{BF}^\dagger(L_{n,\infty})\in{\rm Sel}_{{\tt Gr},\emptyset}(\cK_n^\ac,\mathbf{T}^{\dagger,{\rm cyc}})$ is finite at the places above $\pp$, while at $\ppbar$ chasing through the definitions we see that
\[
\mathfrak{Der}_{-}({\rm loc}_{\ppbar}(\mathcal{BF}^\dagger(L_{n,\infty}))=\beta_n^\dagger(\mathds{1}),
\] 
thus concluding the proof of the height formula (\ref{eq:rubin-ht}).
\end{proof}

%

\section{Big Heegner points}\label{sec:HP} \label{sec:bigHP}

In this section, we recall the construction of big Heegner points and classes \cite{howard-invmath,LV} with some complements.  

Fix a prime $p>3$ and a positive integer $N$ prime to $p$. Let $\cK$ be an imaginary quadratic field with ring of integers $\cO_\cK$ and  
discriminant $-D_\cK<0$ prime to $Np$, and write 
\[
N=N^+N^-
\]
with $N^+$ (resp. $N^-$) divisible only by primes which are split (resp. inert) in $\cK$. Throughout, we assume the following \emph{generalized Heegner hypothesis}: 
\begin{equation}\label{ass:gen-H}
\textrm{$N^-$ is the squarefree product of an even number of primes,}\tag{gen-H}
\end{equation}
and fix an integral ideal $\mathfrak{N}^+$ of $\cK$ with $\cO_\cK/\mathfrak{N}^+\simeq\bZ/N^+\bZ$.

\subsection{Towers of Shimura curves}\label{subsec:Sh}

Let $B/\bQ$ be an 
indefinite quaternion algebra of discriminant $N^-$. We fix a $\Q$-algebra embedding
$\iota_\cK:\cK\hookrightarrow B$, which we shall use to identify $\cK$ with a subalgebra of $B$.
Let $z\mapsto\overline{z}$ be  the non-trivial automorphism of $\cK$, and choose a basis $\{1,j\}$
of $B$ over $\cK$ such that:
\begin{itemize}
	\item $j^2=\beta\in\Q^\times$ with $\beta<0$ and $jt=\bar tj$ for all $t\in \cK$,
	\item $\beta\in (\Z_q^\times)^2$ for $q\mid pN^+$, and $\beta\in\Z_q^\times$ for $q\mid D_\cK$.
\end{itemize}

Fix a square-root $\delta=\sqrt{-D_\cK}$, and define $\boldsymbol{\theta}\in \cK$ by
\begin{equation}\label{eq:vartheta}
\boldsymbol{\theta}:=\frac{D_\cK'+\delta}{2},\quad\textrm{where}\;\;
D_\cK':=\left\{
\begin{array}{ll}
D_\cK &\textrm{if $2\nmid D_\cK$,}\\[0.1cm]
D_\cK/2 &\textrm{if $2\mid D_\cK$,}
\end{array}
\right.
\end{equation}
so that $\mathcal O_\cK=\Z+\boldsymbol{\theta}\bZ$. For every prime $q\mid pN^+$,
define the isomorphism $i_q:B_q:=B\otimes_\Q\Q_q \simeq \M_2(\Q_q)$ by
\[
i_q(\boldsymbol{\theta})=\mat{\mathrm{Tr}(\boldsymbol{\theta})}{-\mathrm{Nm}(\boldsymbol{\theta})}10,
\quad\quad
i_q(j)=\sqrt\beta\mat{-1}{\mathrm{Tr}(\boldsymbol{\theta})}01,
\]
where $\mathrm{Tr}$ and $\mathrm{Nm}$ are the reduced trace and norm maps on $B$. 
For primes $q\nmid Np$, we fix any isomorphism $i_q:B_q\simeq \M_2(\Q_q)$
with $i_q(\mathcal O_\cK\otimes_\Z\Z_q)\subset\M_2(\Z_q)$. 

Let $\hat{\bZ}$ denote the profinite completion of $\bZ$, and for any abelian group $M$ set $\hat{M}=M\otimes_{\bZ}\hat{\bZ}$. 
For each $r\geqslant 0$, let $R_{r}$ be the Eichler order of $B$ of level $N^+p^r$ with respect to 
the isomorphisms $\{i_q:B_q\simeq{\rm M}_2(\Q_q)\}_{q\nmid N^-}$, and let $U_{r}\subset\hat{R}_{r}^\times$ be the compact open
subgroup 
\[
U_{r}:=\left\{(x_q)_q\in\hat{R}_{r}^\times\;\colon\;i_p(x_p)\equiv\mat 1*0*\pmod{p^r}\right\}.
\]

Consider the double coset spaces 
\begin{equation}\label{def:gross-curve}
 X_{r}=B^\times\backslash\bigl(\Hom_\Q(\cK,B)\times\hat{B}^\times/U_{r}\bigr),
\end{equation}
where $b\in B^\times$ acts on $(\Psi,g)\in\Hom_\Q(\cK,B)\times\hat B^\times$ by
\[
b\cdot(\Psi,g)=(b\Psi b^{-1},bg),
\]
and $U_{r}$ acts on $\hat{B}^\times$ by right multiplication. As is well-known (see e.g. \cite[\S\S{2.1-2}]{LV}), $ X_{r}$ can be identified 
with a set of algebraic points on the Shimura curve with complex uniformization
\[
X_{r}(\bC)=B^\times\backslash\bigl(\Hom_\Q(\bC,B)\times\hat{B}^\times/U_{r}\bigr).
\]

Let ${\rm rec}_\cK:\cK^\times\backslash\hat{\cK}^\times\rightarrow{\rm Gal}(\cK^{\rm ab}/\cK)$ be the reciprocity map of class field theory. By Shimura's reciprocity law, if $P\in X_{r}$
is the class of a pair $(\Psi,g)$, then $\sigma\in{\rm Gal}(\cK^{\rm ab}/\cK)$ acts on $P$ by
\[
P^{\sigma}:=[(\Psi,\hat{\Psi}(a)g)],
\]
where $a\in \cK^\times\backslash\hat{\cK}^\times$ is such that ${\rm rec}_\cK(a)=\sigma$, and $\hat\Psi:\hat \cK\rightarrow\hat B$ is the adelization of $\Psi$.  We extend this to an action of 
$G_\cK:={\rm Gal}(\overline{\Q}/\cK)$ in the obvious manner. 

The curves $ X_{r} $ are also equipped
with natural actions of Hecke operators $T_\ell$ for $\ell\nmid Np$, $U_\ell$ for $\ell\vert Np$, and diamond
operators $\langle d \rangle$ for $d\in(\Z/p^r\Z)^\times$, as described in e.g. \cite[\S{2.4}]{LV} and \cite[\S{2.1}]{ChHs2}.

\subsection{Compatible systems of Heegner points}\label{subsec:construct}

For each $c\geqslant 1$, let $\cO_c=\Z+c\cO_\cK$ be the order of $\cK$ of conductor $c$ and denote
by $H_c$ the ring class field of $\cK$ of conductor $c$, so that ${\rm Pic}(\cO_c)\simeq{\rm Gal}(H_c/\cK)$ by class field theory. In particular, $H_1$ is the Hilbert class field of $\cK$.

\begin{defn}\label{def:HP}
	A point $P\in  X_{r}$ is a \emph{Heegner point of conductor $c$}
	if it is the class of a pair $(\Psi,g)$ with
	\[
	\Psi(\cO_c)=\Psi(\cK)\cap(B\cap g\hat{R}_{r}g^{-1})
	\]
	and
	\[
	\Psi_p\left((\cO_c\otimes\Z_p)^\times\cap(1+p^r\cO_c\otimes\Z_p)^\times\right)
	=\Psi_p\left((\cO_c\otimes\Z_p)^\times\right)\cap g_pU_{r,p}g_p^{-1},
	\]
	where $\Psi_p$ and $U_{r,p}$ denote the $p$-components of $\Psi$ and $U_{r}$, respectively. 
\end{defn}

For each prime $q\neq p$ define
\begin{itemize}
	\item{} $\varsigma_q=1$, if $q\nmid N^+$,
	\item{} $\varsigma_q=\delta^{-1}\begin{pmatrix}\boldsymbol{\theta} & \overline{\boldsymbol{\theta}} \\ 1 & 1 \end{pmatrix}
	\in{\rm GL}_2(\cK_{\mathfrak{q}})={\rm GL}_2(\Q_q)$, if
	$q=\mathfrak{q}\overline{\mathfrak{q}}$ splits with $\mathfrak{q}\mid\mathfrak{N}^+$,
\end{itemize}
and for each $s\geqslant 0$, let
\begin{itemize}
	\item{} $\varsigma_p^{(s)}=\begin{pmatrix}\boldsymbol{\theta}&-1\\1&0\end{pmatrix}\begin{pmatrix}p^s&0\\0&1\end{pmatrix}
	\in{\rm GL}_2(\cK_{\mathfrak{p}})={\rm GL}_2(\Q_p)$,
	if $p=\mathfrak{p}\overline{\mathfrak{p}}$ splits in $\cK$,
	\item{}
	$\varsigma_p^{(s)}=\begin{pmatrix}0&1\\-1&0\end{pmatrix}\begin{pmatrix}p^s&0\\0&1\end{pmatrix}$, if $p$ is inert in $\cK$.
\end{itemize}

\begin{rem} We shall ultimately assume that $p$ splits in $\cK$, but it is worth-noting that, just as in \cite{howard-invmath, LV}, the constructions in this section also allow the case $p$ inert in $\cK$.
\end{rem}

Set $\varsigma^{(s)}:=\varsigma_p^{(s)}\prod_{q\neq p}\varsigma_q$,  which we view as an element in $\hat{B}^\times$ via the isomorphisms 
$\{i_q:B_q\simeq{\rm M}_2(\Q_q)\}_{q\nmid N^-}$ introduced in $\S\ref{subsec:Sh}$. With the $\bQ$-algebra embedding $\iota_\cK:\cK\hookrightarrow B$ fixed there, 
one easily checks  that for all $s\geqslant r$ the points
\[
{P}_{s,r}^{}:=[(\iota_\cK,\varsigma^{(s)})]\in X_{r}
\]
are Heegner points of conductor $p^{s}$ in the sense of Definition~\ref{def:HP} with the following properties:
\begin{itemize}
	\item{} \emph{Field of definition}: ${P}_{s,r}\in H^0(H_{p^s}({\mu}_{p^r}),{X}_{r})$.
	\item{} \emph{Galois equivariance}: For all $\sigma\in{\rm Gal}(H_{p^s}({\mu}_{p^r})/H_{p^s})$,
	\[
	{P}_{s,r}^\sigma=\langle\vartheta(\sigma)\rangle\cdot {P}_{s,r},
	\]
	where $\vartheta:{\rm Gal}(H_{p^s}({\mu}_{p^r})/H_{p^{s}})\rightarrow\Z_p^\times/\{\pm{1}\}$ is such that
	$\vartheta^2=\varepsilon_{\rm cyc}$.
	\item{} \emph{Horizontal compatibility}: If $s\geqslant r> 1$, then
	\[
	\sum_{\sigma\in{\rm Gal}(H_{p^s}({\mu}_{p^r})/H_{p^{s-1}}({\mu}_{p^r}))}
	{\alpha}_r({P}_{s,r}^{{\sigma}})
	=U_p\cdot{P}_{s,r-1},
	\]
	where ${\alpha}_r: X_{r}\rightarrow {X}_{r-1}$ is the map
	induced by the inclusion $U_{r}\subset U_{r-1}$.
	\item{} \emph{Vertical compatibility}: If $s\geqslant r\geqslant 1$, then
	\[
	\sum_{\sigma\in{\rm Gal}(H_{p^s}({\mu}_{p^r})/H_{p^{s-1}}({\mu}_{p^r}))}
	{P}_{s,r}^{{\sigma}}
	=U_p\cdot{P}_{s-1,r}.
	\]
\end{itemize}
(See \cite[Thm.~1.2]{cas-longo} and the references therein.)

\subsection{Big Heegner points}\label{subsec:bigHP}

Let $\mathbb{B}_r$ the $\bZ_p$-algebra generated by the Hecke operators 
$T_\ell$, $U_\ell$, and $\langle a\rangle$ acting on the Shimura curve ${X}_{r}$ from $\S$\ref{subsec:Sh}, 
let  $\mathfrak{h}_{r}$ be the $\bZ_p$-algebra generated by the usual Hecke operators $T_\ell$,
$U_\ell$, and $\langle a\rangle$ acting on the space $S_2(\Gamma_{0,1}(N,p^r))$ 
of classical modular form of level $\Gamma_{0,1}(N,p^r):=\Gamma_0(N)\cap\Gamma_1(p^r)$, 
and let $\mathbb{T}^{N^-}_{N,r}$ be the quotient of $\mathfrak{h}_r$ acting faithfully on the 
subspace of $S_2(\Gamma_{0,1}(N,p^r))$ consisting of $N^-$-new forms. 

The Jacquet--Langlands correspondence 
yields $\bZ_p$-algebra isomorphisms 
\begin{equation}\label{eq:JL}
\mathbb{B}_r\simeq\mathbb{T}^{N^-}_{N,r}
\end{equation} 
(see e.g. \cite[\S{2.4}]{HMI}). In particular, letting $e_{\rm ord}=\lim_{n\to\infty}U_p^{n!}$ be Hida's ordinary projector, the $\bZ_p$-module 
\[
\mathfrak{D}_{r}^{\rm ord}:=e_{\rm ord}({\rm Div}({X}_{r})\otimes_{\Z}\Z_p)
\]
is naturally endowed with an action of $\mathbb{T}_{r}^{\rm ord}:=e_{\rm ord}\mathbb{T}^{N^-}_{N,r}$.  
 
Denote by $\mathbb T_{r}^\dagger$ be the free $\mathbb T_{r}^{\rm ord}$-module of rank one 
equipped with the Galois action via the inverse of the critical character $\Theta$, and set $\mathfrak{D}_{r}^\dagger:=\mathfrak{D}_{r}^{\rm ord}\otimes_{\mathbb{T}_{r}^{\rm ord}}\mathbb T_{r}^\dagger$. 

Let ${P}_{s,r}\in{X}_{r}$ be the Heegner point of conductor $p^s$ ($s\geqslant r$) constructed in $\S$\ref{subsec:construct}, and denote by $\mathcal{P}_{s,r}$ the image of
$e_{\rm ord}{P}_{s,r}^{}$ in $\mathfrak{D}_{r}^{\rm ord}$.
It follows from the Galois-equivariance property of ${P}_{s,r}$ that
\[
\mathcal{P}_{s,r}^\sigma=\Theta(\sigma)\cdot\mathcal{P}_{s,r}
\]
for all $\sigma\in{\rm Gal}(H_{p^s}({\mu}_{p^r})/H_{p^{s}})$ (see \cite[\S{7.1}]{LV}),
and hence $\mathcal{P}_{s,r}$ defines an element
\begin{equation}\label{eq:pt-dag}
\mathcal{P}_{s,r}\otimes\zeta_r\in\rH^0(H_{p^{s}},\mathfrak{D}_{r}^\dagger).
\end{equation}
Let ${\rm Pic}({X}_{r})$ be the Picard variety of ${X}_{r}$, and set
\[
\mathfrak{J}_{r}^{\rm ord}:=e_{\rm ord}({\rm Pic}({X}_{r})\otimes_{\Z}\Z_p),\quad\quad\mathfrak{J}_{r}^\dagger:=\mathfrak{J}_{r}^{\rm ord}\otimes_{\mathbb{T}_r^{\rm ord}}\mathbb{T}_r^\dagger.
\]
Since the $U_p$-operator has degree $p$, taking ordinary parts yields an isomorphism $\mathfrak{D}_r^{\rm ord}\simeq\mathfrak{J}_r^{\rm ord}$, and so we may also view (\ref{eq:pt-dag}) as $\mathcal{P}_{s,r}\otimes\zeta_r\in\rH^0(H_{p^{s}},\mathfrak{J}_{r}^\dagger)$.

Let $t\geqslant 0$, and denote by $\mathfrak{G}_{H_{p^t}}$ the Galois group of the maximal
extension of $H_{p^t}$ unramified outside the primes above $pN$. Consider the twisted Kummer map
\[
{\rm Kum}_r:\rH^0(H_{p^t},\mathfrak{J}_{r}^\dagger)
\rightarrow\rH^1(\mathfrak{G}_{H_{p^t}},{\rm Ta}_p(\mathfrak{J}_{r}^\dagger))
\]
explicitly defined in \cite[p.~101]{howard-invmath}. This map is equivariant for the Galois and $U_p$ actions, and hence by horizontal compatibility the classes
\begin{equation}\label{eq:HP-r}
\mathfrak{X}_{p^t,r}:={\rm Kum}_r({\rm Cor}_{H_{p^{r+t}/H_{p^t}}}(\mathcal{P}_{r+t,r}\otimes\zeta_r))
\end{equation}
satisfy ${\alpha}_{r,*}(\mathfrak{X}_{p^t,r})=U_p\cdot\mathfrak{X}_{p^t,r-1}$ for all $r>1$, where 
\[
{\alpha}_{r,*}:\rH^1(\mathfrak{G}_{H_{p^t}},{\rm Ta}_p(\mathfrak{J}_{r}^\dagger))\rightarrow\rH^1(\mathfrak{G}_{H_{p^t}},{\rm Ta}_p(\mathfrak{J}_{r-1}^\dagger))
\]
is the map induced by the covering ${X}_r\rightarrow{X}_{r-1}$ by Albanese functoriality.

Now let $\F\in\cR[[q]]$ be a Hida family of tame level $N$. To define big Heegner points attached to $\F$ from the system of Heegner classes (\ref{eq:HP-r}) for varying $r$, we need to recall the following result realizing the big Galois representation $T_\F$ attached to $\F$ in the \'etale cohomology of the $p$-tower of Shimura curves
\[
\cdots\rightarrow{X}_r\rightarrow{X}_{r-1}\rightarrow\cdots
\]
(rather than classical modular curves, as implicitly taken in $\S\ref{sec:selmer}$).

Let $\kappa_\cR:=\cR/\mathfrak{m}_\cR$ be the residue field of $\cR$, and denote by $\bar{\rho}_{\F}:G_\bQ\rightarrow{\rm GL}_2(\kappa_\cR)$ the associated (semi-simple) residual representation. Set
\[
\mathbb{T}_\infty^{\rm ord}:=\varprojlim_r\mathbb{T}_r^{\rm ord}.
\] 
By (\ref{eq:JL}) (see also the discussion in \cite[\S{5.3}]{LV}), there is a maximal ideal $\mathfrak{m}\subset\mathbb{T}^{\rm ord}_\infty$ associated with $\bar{\rho}_{\F}$, and $\F$ corresponds to a minimal prime in the localization $\mathbb{T}^{\rm ord}_{\infty,\mathfrak{m}}$.

\begin{thm}\label{thm:helm}
Assume that 
\begin{itemize}
\item[(i)] $\bar{\rho}_{\F}$ is irreducible and $p$-distinguished,
\item[(ii)] $\bar{\rho}_{\F}$ is ramified at every prime $\ell\vert N^-$ with $\ell\equiv\pm{1}\pmod{p}$,
\end{itemize}
and let $\mathfrak{m}\subset\mathbb{T}_\infty^{\rm ord}$ be the maximal ideal associated with $\bar{\rho}_{\F}$. Then the module 
\[
\mathbf{Ta}_{\mathfrak{m}}^{\rm ord}:=\biggl(\varprojlim_r{\rm Ta}_p(\mathfrak{J}_r^{\rm ord})\biggr)\otimes_{\mathbb{T}_\infty^{\rm ord}}\mathbb{T}_{\infty,\mathfrak{m}}^{\rm ord}
\]
is free of rank $2$ over $\mathbb{T}_{\infty,\mathfrak{m}}^{\rm ord}$, and if $\F$ corresponds to the minimal prime $\mathfrak{a}\subset\mathbb{T}_{\infty,\mathfrak{m}}^{\rm ord}$, then there is an isomorphism
\[
T_\F\simeq\mathbf{Ta}_{\mathfrak{m}}^{\rm ord}\otimes_{\mathbb{T}_{\infty,\mathfrak{m}}^{\rm ord}}\mathbb{T}_{\infty,\mathfrak{m}}^{\rm ord}/\mathfrak{a}
\]
as $(\mathbb{T}_{\infty,\mathfrak{m}}^{\rm ord}/\mathfrak{a})[G_\bQ]$-modules.
\end{thm}

\begin{proof}
This is shown in \cite[Thm.~3.1]{Fouquet} assuming  the ``mod $p$ multiplicity one'' hypothesis in [\emph{loc.cit.}, Prop.~3.7]. Since by \cite[Cor.~8.11]{helm} that hypothesis is ensured by our ramification condition on $\bar\rho_{\F}$, the result follows.  
\end{proof}

Let $\mathfrak{m}\subset\mathbb{T}_\infty^{\rm ord}$ be a maximal ideal satisfying the hypotheses of Theorem~\ref{thm:helm}, and suppose that the Hida family $\F$ corresponds to a minimal prime of $\mathbb{T}_{\infty,\mathfrak{m}}^{\rm ord}$, so by Theorem~\ref{thm:helm} there is a quotient map $\mathbf{Ta}_{\mathfrak{m}}^{\rm ord}\rightarrow T_\F$. Note also that immediately from the definitions there are natural maps 
${\rm Ta}_p(\mathfrak{J}_r^\dagger)\rightarrow\mathbf{Ta}_{\mathfrak{m}}^{\rm ord}\otimes\Theta^{-1}\rightarrow T_\F^\dagger$.

\begin{defn}
	The \emph{big Heegner point} of conductor $p^t$ is the class 
	\[
	\mathfrak{X}_{p^t}\in\rH^1(H_{p^t},T_{\F}^\dagger)
	\]
	given by the image of $\varprojlim_rU_p^{-r}\cdot\mathfrak{X}_{p^t,r}$ under the composite map
	\[
	\varprojlim_r\rH^1(\mathfrak{G}_{H_{p^t}},{\rm Ta}_p(\mathfrak{J}_{r}^\dagger))
	\rightarrow\rH^1(\mathfrak{G}_{H_{p^t}},\mathbf{Ta}_{\mathfrak{m}}^{\rm ord}\otimes\Theta^{-1})
	\rightarrow\rH^1(H_{p^t},T_{\F}^\dagger).
	\]
\end{defn}

We conclude this section with the following result essentially due to Howard, showing that the big Heegner points are Selmer classes under mild hypotheses.

\begin{prop}\label{prop:HPinSel}
Assume that $\bar{\rho}_{\F}$ is ramified at every prime $\ell\vert N^-$. 
Then the classes $\mathfrak{X}_{p^t}$ lie in ${\rm Sel}_{\tt Gr}(H_{p^t},T_{\F}^\dagger)$.
\end{prop}

\begin{proof}
The argument in \cite[Prop.~2.4.5]{howard-invmath} (see also \cite[Prop.~10.1]{LV}) shows that for every prime $w$ of $H_{p^t}$ the localization
${\rm loc}_w(\mathfrak{X}_{p^t})$ lies in the subspace $\rH_{\tt Gr}^1(H_{p^t,w},T_{\F}^\dagger)\subset\rH^1(H_{p^t,w},T_{\F}^\dagger)$ defining ${\rm Sel}_{\tt Gr}(H_{p^t},T_{\F}^\dagger)$, except when $w\vert\ell\vert N^-$, in which case it is shown that
\[
{\rm loc}_{w}(\mathfrak{X}_{p^t})\in{\rm ker}\biggl\{\rH^1(H_{p^t,w},T_{\F}^\dagger)\rightarrow\frac{\rH^1(H_{p^t,w}^{\rm ur},T_{\F}^\dagger)}{\rH^1(H_{p^t,w}^{\rm ur},T_{\F}^\dagger)_{\rm tors}}\biggr\},
\]
where $\rH^1(H_{p^t,w}^{\rm ur},T_{\F}^\dagger)_{\rm tors}$ denotes the $\cR$-torsion submodule of $\rH^1(H_{p^t,w}^{\rm ur},T_{\F}^\dagger)$. However, 
such primes $\ell$ are inert in $\cK$, so 
$H_{p^t,w}=\cK_\ell$, and since our hypothesis on $\bar\rho_\F$ implies that $\rH^1(\cK_\ell^{\rm ur},T_{\F}^\dagger)$ is $\cR$-torsion free (see e.g. \cite[Lem.~3.12]{buy-bigHP}), the result follows.
\end{proof}


Recall that $\cK_\infty^\ac=\cup_n \cK_n^\ac$ is the anticyclotomic $\bZ_p$-extension of $\cK$, with $\cK_n^\ac$ denoting the subextension of $\cK_\infty^\ac$ with $[\cK_n^\ac:\cK]=p^n$. Similarly as in \cite[\S{3.3}]{howard-invmath} and \cite[\S{10.3}]{LV}, 
we set
\[
\mathfrak{Z}_n:={\rm Cor}_{H_{p^t}/\cK^\ac_n}(U_p^{-t}\cdot\mathfrak{X}_{p^t})\in\rH^1(\cK_n^\ac,T_{\F}^\dagger),
\]
where $t\gg 0$ is chosen so that $\cK_n^\ac\subset H_{p^t}$. By horizontal compatibility, the definition of $\mathfrak{Z}_n$ is independent of the choice of $t$, and for varying $n$ they define a system
\[
\mathfrak{Z}_\infty:=\{\mathfrak{Z}_n\}_n\in\varprojlim_n\rH^1(\cK_n^\ac,T_{\F}^\dagger)\simeq\rH_{}^1(\cK,\mathbf{T}^{\dagger,\ac})
\]
which is not $\cR[[\Gamma_\cK^\ac]]$-torsion  
by the work of Cornut--Vatsal \cite{CV-dur} (see also \cite[Cor.~3.1.2]{howard-invmath}).

\section{Main results} \label{sec:main}

In this section we conclude the proof of the main results of this paper.   
Fix a prime $p>3$ and let 
\[
\F=\sum_{n=1}^\infty\boldsymbol{a}_nq^n\in\cR[[q]]
\] 
be a primitive 
Hida family of tame level $N$, and let $\cK$ be an imaginary quadratic field of discriminant prime to $Np$ satisfying the generalized Heegner hypothesis (\ref{ass:gen-H}) relative to $N$. 
Our results will require some of the technical hypotheses below, which we record here for our later reference.

\begin{itemize}
\item[(h0)] $\cR$ is regular,
\item[(h1)]{} some specialization $\F_\phi$ is the $p$-stabilization of a newform $f\in S_2(\Gamma_0(N))$,
\item[(h2)] $\bar{\rho}_{\F}$ is irreducible, 
\item[(h3)] $N$ is squarefree,
\item[(h4)] $N^-\neq 1$, 
\item[(h5)] $\bar{\rho}_{\F}$ is ramified at every prime $\ell\vert N^-$,
\item[(h6)] $p$ splits in $\cK$.
\end{itemize}
As usual, here $N^-$ denotes the largest factor of $N$ divisible only by primes which are inert in $\cK$. 

\subsection{Converse to a theorem of Howard}

As shown by Howard \cite[\S\S{2.3-4}]{howard-invmath},  
for varying $c$ prime to $N$ the big Heegner points $\mathfrak{X}_c\in\rH^1(H_c,T_{\F}^\dagger)$ form an anticyclotomic Euler system for $T_{\F}^\dagger$. Setting
\[
\mathfrak{Z}_0:={\rm Cor}_{H_1/\cK}(\mathfrak{X}_1)\in\rH^1(\cK,T_{\F}^\dagger),
\]
Kolyvagin's methods thus yield a proof of the implication
\[
\mathfrak{Z}_0\not\in{\rm Sel}_{\tt Gr}(\cK,T_{\F}^\dagger)_{\rm tors}
\quad\Longrightarrow\quad
{\rm rank}_\cR\;{\rm Sel}_{\tt Gr}(\cK,T_{\F}^\dagger)=1,
\]
where the subscript ${\rm tors}$ denotes the $\cR$-torsion submodule 
(see \cite[Cor.~3.4.3]{howard-invmath}). 

In the spirit of Skinner's celebrated converse to the theorem of Gross--Zagier and Kolyvagin \cite{skinner}, in this section we  prove a result in the converse direction (see Theorem~\ref{thm:converse-How} below).  

Following the approach of \cite{wan}, this will be deduced from progress on the Iwasawa main conjecture for big Heegner points (\cite[Conj.~3.3.1]{howard-invmath}) 
exploiting the non-triviality of $\mathfrak{Z}_\infty$. 





\begin{thm}\label{thm:HP-MC}
Assume hypotheses {\rm (h0)--(h6)}. Then both $X_{\tt Gr}(\cK_\infty^\ac,A_\F^\dagger)$ and ${\rm Sel}_{\tt Gr}(\cK,\Tc)$ have $\cR[[\Gamma_\cK^\ac]]$-rank one, and 
	\[
	{\rm Char}_{\cR[[\Gamma_\cK^\ac]]}(X_{\tt Gr}(\cK_\infty^\ac,A_\F^\dagger)_{\rm tors})=
	{\rm Char}_{\cR[[\Gamma_\cK^\ac]]}\biggl(\frac{{\rm Sel}_{\tt Gr}(\cK,\Tc)}{\cR[[\Gamma_\cK^\ac]]\cdot\mathfrak{Z}_\infty}\biggr)^2,
	\]
where the subscript {\rm tors} denotes the $\cR[[\Gamma_\cK^\ac]]$-torsion submodule.
\end{thm}

\begin{proof}
Since $\mathfrak{Z}_\infty$ is not $\cR[[\Gamma_\cK^\ac]]$-torsion,  
 part (iii) of  \cite[Thm.~B]{Fouquet} implies that $X_{\tt Gr}(\cK_\infty^\ac,A_\F^\dagger)$ and ${\rm Sel}_{\tt Gr}(\cK,\Tc)$ have $\cR[[\Gamma_\cK^\ac]]$-rank one, and that the divisibility
\begin{equation}\label{eq:fouquet}
	{\rm Char}_{\cR[[\Gamma_\cK^\ac]]}(X_{\tt Gr}(\cK_\infty^\ac,A_\F^\dagger)_{\rm tors})\supset
	{\rm Char}_{\cR[[\Gamma_\cK^\ac]]}\biggl(\frac{{\rm Sel}_{\tt Gr}(\cK,\Tc)}{\cR[[\Gamma_\cK^\ac]]\cdot\mathfrak{Z}_\infty}\biggr)^2
\end{equation}
holds in $\cR[[\Gamma_\cK^\ac]]$. Concerning the additional hypotheses in Fouquet's result, we note that:
\begin{itemize}
		\item Assumption~3.4, that $\bar\rho_\F$ is irreducible, is our (h2),
		\item Assumption~3.5, that $\bar\rho_\F$ is $p$-distinguished, follows from (h1) (see \cite[Rem.~7.2.7]{KLZ2}),
		\item Assumption~3.10, that the tame character of $\F$ admits a square-root, is satisfied by (h1),
		\item Assumption~5.10, that all primes $\ell\vert N$ for which $\bar\rho_\F$ is not ramified have infinite decomposition group in $\cK_\infty^\ac/\cK$, is a reformulation of (h5),
		\item Assumption~5.13, that $\bar{\rho}_{\F}\vert_{G_\cK}$ is irreducible, follows from (h2), (h4) and (h5) (see \cite[Lem.~2.8.1]{skinner}).
\end{itemize}

Let $\phi\in\mathcal{X}_a(\cR)$ be such that $\F_\phi$ is the ordinary $p$-stabilization of a newform $f\in S_2(\Gamma_0(N))$  as in hypothesis (h1). 
Letting $X\supset Y$ stand for the divisibility $(\ref{eq:fouquet})$, by \cite[Thm.~3.4]{cas-BF} (see also \cite[Thm.~1.2]{wan}) we have the equality
\[
X=Y\pmod{{\rm ker}(\phi)\cR[[\Gamma_\cK^\ac]]},
\]
(note that this requires the additional hypotheses (h3) and (h6)), from where the result follows by an application of Lemma~\ref{lem:3.2}. 
\end{proof}

\begin{thm}\label{thm:converse-How}
Assume hypotheses {\rm (h0)--(h6)}. Then the following implication holds: 
\[
{\rm rank}_\cR\;{\rm Sel}_{\tt Gr}(\cK,T_{\F}^\dagger)=1\quad\Longrightarrow\quad
\mathfrak{Z}_0\not\in{\rm Sel}_{\tt Gr}(\cK,T_{\F}^\dagger)_{\rm tors}
\]
where the subscript ${\rm tors}$ denotes the $\cR$-torsion submodule.
\end{thm}

\begin{proof}
Let $\gamma_\ac\in\Gamma_\cK^{\ac}$ be a topological generator. The  restriction map for the extension $\cK_\infty^\ac/\cK$ induces a surjective homomorphism
\[
X_{\tt Gr}(\cK_\infty^\ac,A_\F^\dagger)/(\gamma_\ac-1)X_{\tt Gr}(\cK_\infty^\ac,A_\F^\dagger)\twoheadrightarrow X_{\tt Gr}(\cK,A_{\F}^\dagger)
\]
with $\cR$-torsion kernel. Since $X_{\tt Gr}(\cK,A_{\F}^\dagger)$ and ${\rm Sel}_{\tt Gr}(\cK,T_{\F}^\dagger)$ have the same $\cR$-rank by Lemma~\ref{lem:eq-ranks}, we thus see from  Theorem~\ref{thm:HP-MC} that our assumption implies that  
	\[
	(\gamma_\ac-1)\nmid{\rm Char}_{\cR[[\Gamma_\cK^\ac]]}\biggl(\frac{{\rm Sel}_{\tt Gr}(\cK,\Tc)}{\cR[[\Gamma_\cK^\ac]]\cdot\mathfrak{Z}_\infty}\biggr).
	\] 
	By \cite[Cor.~3.8]{SU} (with $F=\cK_\infty^\ac)$, it follows that 
	the image of $\mathfrak{Z}_\infty$ in ${\rm Sel}_{\tt Gr}(\cK,\Tc)/(\gamma_{\ac}-1){\rm Sel}_{\tt Gr}(\cK,\Tc)$ is not $\cR$-torsion; since this image is sent to $\mathfrak{Z}_0$ under the natural injection
	\[
	{\rm Sel}_{\tt Gr}(\cK,\Tc)/(\gamma_{\ac}-1){\rm Sel}_{\tt Gr}(\cK,\Tc)\hookrightarrow{\rm Sel}_{\tt Gr}(\cK,T_{\F}^\dagger),
	\]
the result follows.
\end{proof}

\begin{rem}
Replacing the application of \cite[Thm.~3.4]{cas-BF} or \cite[Thm.~1.2]{wan} in the proof of Theorem~\ref{thm:HP-MC} by an application of \cite[Thm.~1.2]{BCK} or \cite[Thm.~1.1.5]{zanarella} the above argument gives a proof of Theorems~\ref{thm:HP-MC}  
and \ref{thm:converse-How}
with hypotheses (h3)--(h6) replaced by ``Hypothesis~$\heartsuit$'' from \cite{zhang-kolyvagin}, i.e.,  
letting ${\rm Ram}(\bar{\rho}_\F)$ be the set of primes $\ell\Vert N$ such that 
$\bar{\rho}_\F$ is ramified at $\ell$: 
	\begin{itemize}
		\item ${\rm Ram}(\bar{\rho}_\F)$ contains all primes $\ell\Vert N^+$, and all primes  $\ell\vert N^-$ such that $\ell\equiv\pm 1 \pmod{p}$, 
		\item If $N$ is not squarefree, then ${\rm Ram}(\bar{\rho}_\F)\neq\emptyset$, and either ${\rm Ram}(\bar{\rho}_\F)$ contains a prime $\ell\vert N^-$ or there are at least two primes $\ell\Vert N^+$,
		\item If $\ell^2\vert N^+$, then $\rH^0(\bQ_\ell,\bar{\rho}_\F)=\{0\}$,
	\end{itemize}
and the assumption that $\bar{\rho}_{\F}$ is surjective and $\boldsymbol{a}_p\not\equiv\pm{1}\pmod{p}$.
\end{rem}

\subsection{Iwasawa--Greenberg main conjectures}

Now we can upgrade the main result 
 in \cite{wanIMC} towards the Iwasawa--Greenberg main conjecture for $\mathscr{L}_\pp(\F/\cK)$ 
to a proof of the full equality. 



\begin{thm}\label{thm:3-IMC-BDP}
Assume hypotheses {\rm (h0)--(h6)}.  
Then $X_{\emptyset,0}(\cK_\infty,A_\F)$ 
is $\cR[[\Gamma_\cK]]$-torsion and
\begin{equation}\label{eq:3-IMC-bdp}
{\rm Char}_{\cR[[\Gamma_\cK]]}(X_{\emptyset,0}(\cK_\infty,A_\F))\cdot\cR^{\rm ur}[[\Gamma_\cK]]=(\mathscr{L}_\pp(\F/\cK))\nonumber
\end{equation}
as ideals in $\cR^{\rm ur}[[\Gamma_\cK]]$.
\end{thm}

\begin{proof}
Clearly (see \cite[Lem.~1.2]{Rubin-ES}), it suffices to show that the twisted module $X_{\emptyset,0}(\cK_\infty,A_\F^\dagger)$ is $\cR[[\Gamma_\cK]]$-torsion, with characteristic ideal generated by ${\rm tw}_{\Theta^{-1}}(\mathscr{L}_\pp(\F/\cK))$ after extension of scalars to $\cR^{\rm ur}[[\Gamma_\cK^\ac]]$. This in turn can be shown by a similar argument to that in the  proof of Theorem~\ref{thm:HP-MC}, so we shall be rather brief. 

Taking $\phi\in\mathcal{X}_a(\cR)$ such that $\F_\phi$ is the ordinary $p$-stabilization of a newform $f\in S_2(\Gamma_0(N))$, we deduce that $X_{\emptyset,0}(\cK_\infty^\ac,A_\F^{\dagger,\ac})$ is $\cR[[\Gamma_\cK^\ac]]$-torsion and that the equality 
as ideals in $\cR^{\rm ur}[[\Gamma_\cK^\ac]]$
\begin{equation}\label{eq:BDP-IMC}
{\rm Char}_{\cR[[\Gamma_\cK^\ac]]}(X_{\emptyset,0}(\cK_\infty^\ac,A_\F^{\dagger,\ac}))\cdot\cR^{\rm ur}[[\Gamma_\cK^\ac]]=(\mathscr{L}_\pp^{\tt BDP}(\F/\cK)^2)
\end{equation}
holds by applying Lemma~\ref{lem:3.2} to the combination of the divisibility in \cite[Thm.~1.1]{wanIMC} (projected under $\Gamma_\cK\twoheadrightarrow\Gamma_\cK^\ac$) with the equality in \cite[Thm.~3.4]{cas-BF} for $f=\F_\phi$.  
The $\cR[[\Gamma_\cK]]$-torsionness of $X_{\emptyset,0}(\cK_\infty,A_\F^\dagger)$ then follows from that of $X_{\emptyset,0}(\cK_\infty^\ac,A_\F^\dagger)$ over $\cR[[\Gamma_\cK^\ac]]$, and using the comparison of $p$-adic $L$-functions in Corollary~\ref{cor:wan-bdp}, the three-variable divisibility  
in \cite[Thm.~1.1]{wanIMC} combined with the equality (\ref{eq:BDP-IMC}) 
yields the desired three-variable equality by another application of 
Lemma~\ref{lem:3.2}. 
\end{proof}

We can now deduce from Theorem~\ref{thm:3-IMC-BDP} the proof of Theorem~A in the Introduction.

\begin{cor}\label{thm:3-IMC}
Assume hypotheses {\rm (h0)--(h6)}. 
Then $X_{\tt Gr}(\cK_\infty,A_\F)$ 
is $\cR[[\Gamma_\cK]]$-torsion, and
\begin{equation}\label{eq:3-IMC}
{\rm Char}_{\cR[[\Gamma_\cK]]}(X_{\tt Gr}(\cK_\infty,A_\F))=(L_p^{\tt Hi}(\F/\cK))\nonumber
\end{equation}	
as ideals in $\cR[[\Gamma_\cK]]\otimes_{\bZ_p}\bQ_p$.
\end{cor}

\begin{proof}
As in the proof of Theorem~\ref{thm:3-IMC-BDP}, it suffices to show that the twisted module $X_{\tt Gr}(\cK_\infty,A_\F^\dagger)$ is $\cR[[\Gamma_\cK]]$-torsion with characteristic ideal generated by ${\rm tw}_{\Theta^{-1}}(L_p^{\tt Hi}(\F/\cK))$, which by the equivalence between main conjectures in Theorem~\ref{thm:2-varIMC},  follows from Theorem~\ref{thm:3-IMC-BDP}. 
\end{proof}

\subsection{Greenberg's nonvanishing conjecture for derivatives}
\label{sec:appl-greenberg}

As in the Introduction, let $-w\in\{\pm{1}\}$ be the generic sign in the functional equation of the $p$-adic $L$-functions $L_p^{\tt MTT}(\F_\phi,s)$ for varying $\phi\in\mathcal{X}_a^o(\cR)$.

By \cite[Cor.~3.4.3 and Eq.~(21)]{howard-invmath}, Howard's horizontal nonvanishing conjecture implies that
\[
{\rm rank}_{\cR}\;{\rm Sel}_{\tt Gr}(\bQ,T_{\F}^\dagger)=\left\{
\begin{array}{ll}
1&\textrm{if $w=1$,}\\
0&\textrm{if $w=-1$.}
\end{array}
\right. 
\]

In the case $w=-1$, a result in the converse direction follows from \cite{SU}: 

\begin{thm}[Skinner--Urban]\label{thm:Gr+1}
Assume that:
\begin{itemize}
\item{} $\bar{\rho}_{\F}$ is irreducible and $p$-distinguished,
\item{} $\F$ has trivial tame character,
\item{} there is a prime $\ell\Vert N$ such that $\bar{\rho}_{\F}$ is ramified at $\ell$.
\end{itemize}
If ${\rm Sel}_{\tt Gr}(\Q,T_{\F}^\dagger)$ is $\cR$-torsion, then $L(\F_\phi,k_\phi/2)\neq 0$ for all but 
finitely many $\phi\in\mathcal{X}_a^o(\cR)$.
\end{thm}

\begin{proof}
Since the $\cR$-modules ${\rm Sel}_{\tt Gr}(\bQ,T_{\F}^\dagger)$ and $X_{\tt Gr}(\bQ,A_{\F}^\dagger)$ have the same rank by Lemma~\ref{lem:eq-ranks}, our hypothesis implies that $X_{\tt Gr}(\bQ,A_{\F}^\dagger)$ is $\cR$-torsion. Thus in particular 
${\rm Sel}_{\tt Gr}(\bQ,A_{\F_\phi}(1-k_\phi/2))$ is finite for all but finitely many $\phi$ as in the statement, and so the result follows from \cite[Thm.~3.6.13]{SU}.
\end{proof}



Our application to Greenberg's nonvanishing conjecture (in the case $w=1$) will build on an $\cR$-adic Gross--Zagier formula for the big Heegner point $\mathfrak{Z}_0$. In fact, we shall prove a formula of this type for the $\cR[[\Gamma_\cK^\ac]]$-adic family $\mathfrak{Z}_\infty$, and deduce the result for $\mathfrak{Z}_0$ by specialization at the trivial character.

Define the cyclotomic $\cR[[\Gamma_\cK^\ac]]$-adic height pairing
\begin{equation}\label{eq:lambda-ht}
\langle,\rangle_{\cK^\ac_\infty,\cR}^{\cyc}:{\rm Sel}_{\tt Gr}(\cK,\Tc)
\otimes_{\cR[[\Gamma_\cK^\ac]]}{\rm Sel}_{\tt Gr}(\cK,\Tc)^\iota
\rightarrow\mathcal{J}^{\rm cyc}\otimes_{\cR}\cR[[\Gamma_\cK^\ac]]\otimes_\cR F_\cR
\end{equation}
by 
\[
\langle a_\infty,b_\infty\rangle^{\rm cyc}_{\cK_\infty^\ac,\cR}=
\varprojlim_n\sum_{\sigma\in{\rm Gal}(\cK_n^\ac/\cK)}\langle a_n,b_n^\sigma\rangle^{\rm cyc}_{\cK_n^\ac,\cR}\cdot\sigma
\]
(using the fact that the $\cR$-adic height pairing $\langle,\rangle_{\cK_n^\ac,\cR}^{\rm cyc}$ have denominators that are bounded independently of $n$), and let the cyclotomic regulator $\mathcal{R}_{\rm cyc}\subset\cR[[\Gamma_\cK^\ac]]\otimes_\cR F_\cR$ 
be the characteristic ideal of the cokernel of $(\ref{eq:lambda-ht})$ (after dividing by the image of $(\gamma_{\rm cyc}-1)$ in $\mathcal{J}^{\rm cyc}$). 




Recall that since we assume that $\cK$ satisfies the generalized Heegner hypothesis (\ref{ass:gen-H}), the constant term $L_{p,0}^{\tt Hi}(\F^\dagger/\cK)_{\ac}$ in the expansion $(\ref{eq:2-exp})$ vanishes by Proposition~\ref{thm:hida-1}. The next result provides a first interpretation of the linear term $L_{p,1}^{\tt Hi}(\F^\dagger/\cK)_{\ac}$.

\begin{thm}\label{thm:3.1.5}
Assume hypotheses {\rm (h0)--(h6)}, and denote by $\mathcal{X}^{}_{\rm tors}$ the characteristic ideal of $X_{\tt Gr}(\cK_\infty^\ac,A_{\F}^\dagger)_{\rm tors}$. Then
\[
\mathcal{R}_{\rm cyc}^{}\cdot\mathcal{X}_{\rm tors}
=(L_{p,1}^{\tt Hi}(\F^\dagger/\cK)_{\ac})
\]
as ideals in $\cR[[\Gamma_\cK^\ac]]\otimes_{\cR}F_\cR$. 
\end{thm}

\begin{proof}
The height formula of Theorem~\ref{thm:rubin-ht} and Lemma~\ref{lem:3.1.1} immediately
yield the equality
\begin{equation}\label{eq:cor-ht}
\mathcal{R}_{\rm cyc}\cdot
{\rm Char}_{\cR[[\Gamma_\cK^\ac]]}\biggl(\frac{{\rm Sel}_{\tt Gr}(\cK,\Tc)}{\cR[[\Gamma_\cK^\ac]]\cdot\mathcal{BF}^{\dagger,\ac}}\biggr)
=(L_{p,1}^{\tt Hi}(\F^\dagger/\cK)_{\ac})\cdot\eta^\iota,
\end{equation}
where $\eta\subset\cR[[\Gamma_\cK^\ac]]$ is the characteristic ideal of
$\rH^1_{\tt Gr}(\cK_{\ppbar},\Tc)/{\rm loc}_{\ppbar}({\rm Sel}_{\tt Gr}(\cK,\Tc))$. 
We shall argue below that $\eta\neq 0$. Global duality yields the exact sequence
\begin{equation}\label{eq:PT2}
0\rightarrow\frac{\rH^1_{\tt Gr}(\cK_{\pp},\Tc)}
{{\rm loc}_{\pp}({\rm Sel}_{\tt Gr}(\cK,\Tc))}
\rightarrow X_{\emptyset,{\tt Gr}}(\cK_\infty^\ac,A_{\F}^\dagger)
\rightarrow X_{\tt Gr}(\cK_\infty^\ac,A_{\F}^\dagger)\rightarrow 0.
\end{equation}
Note that the left-most term in $(\ref{eq:PT2})$ is $\cR[[\Gamma_\cK^\ac]]$-torsion, since by Corollary~\ref{cor:str-Gr} the image of the map ${\rm loc}_{\pp}:{\rm Sel}_{\tt Gr}(\cK,\Tc)\rightarrow \rH^1_{\tt Gr}(\cK_\pp,\Tc)$ is nonzero and the target has $\cR[[\Gamma_\cK^\ac]]$-rank one. By Theorem~\ref{thm:HP-MC}, it follows that the middle term has $\cR[[\Gamma_\cK^\ac]]$-rank one, and by the action of complex conjugation the same is true for $X_{{\tt Gr},\emptyset}(\cK_\infty^\ac,A_{\F}^\dagger)$. Thus the nonvanishing of $\eta$ follows from the analogue of $(\ref{eq:PT2})$ for the prime $\ppbar$ (see (\ref{eq:PT2bar}) below).

By Lemma~\ref{lem:str-rel} the above also shows that $X_{{\tt Gr},0}(\cK_\infty^\ac,A_{\F}^\dagger)$ is $\cR[[\Gamma_\cK^\ac]]$-torsion, and counting ranks in the exact sequence 
\begin{align*}
0\rightarrow {\rm Sel}_{\tt Gr}(\cK,\Tc)\rightarrow{\rm Sel}_{{\tt Gr},\emptyset}(\cK,\Tc)\rightarrow&\frac{\rH^1(\cK_{\overline{\pp}},\Tc)}{\rH^1_{\tt Gr}(\cK_{\overline{\pp}},\Tc)}\\
&\rightarrow
X_{\tt Gr}(\cK_\infty^\ac,A_{\F}^\dagger)
\rightarrow X_{{\tt Gr},0}(\cK_\infty^\ac,A_{\F}^\dagger)\rightarrow 0,
\end{align*} 
we conclude that the first two terms in this sequence have $\cR[[\Gamma_\cK^\ac]]$-rank one. Since the quotient $\rH^1(\cK_{\overline{\pp}},\Tc)/\rH^1_{\tt Gr}(\cK_{\overline{\pp}},\Tc)$ has no $\cR[[\Gamma_\cK^\ac]]$-torsion, it follows that 
\begin{equation}\label{eq:Gr-rel}
{\rm Sel}_{\tt Gr}(\cK,\Tc)=
{\rm Sel}_{{\tt Gr},\emptyset}(\cK,\Tc).
\end{equation}
Taking $\cR[[\Gamma_\cK^\ac]]$-torsion in the analogue of $(\ref{eq:PT2})$ for $\ppbar$, that is
\begin{equation}\label{eq:PT2bar}
0\rightarrow\frac{\rH^1_{\tt Gr}(\cK_{\ppbar},\Tc)}
{{\rm loc}_{\ppbar}({\rm Sel}_{\tt Gr}(\cK,\Tc))}
\rightarrow X_{{\tt Gr},\emptyset}(\cK_\infty^\ac,A_{\F}^\dagger)
\rightarrow X_{\tt Gr}(\cK_\infty^\ac,A_{\F}^\dagger)\rightarrow 0,
\end{equation}
and applying Lemma~\ref{lem:str-rel} and the ``functional equation'' $\mathcal{X}_{\rm tors}^\iota=\mathcal{X}_{\rm tors}$ of \cite[p.~1464]{howard-PhD-I} we obtain
\begin{equation}\label{eq:takechar}
\begin{split}
\eta^\iota\cdot\mathcal{X}_{\rm tors}
&={\rm Char}_{\cR[[\Gamma_\cK^\ac]]}(X_{{\tt Gr},0}(\cK_\infty^\ac,A_{\F}^\dagger)).
\end{split}
\end{equation}
On the other hand, by the equivalence (i)'$\Longleftrightarrow$(ii)' in Theorem~\ref{thm:2-varIMC}, the equality (\ref{eq:BDP-IMC}) implies that
\[
{\rm Char}_{\cR[[\Gamma_\cK^\ac]]}(X_{{\tt Gr},0}(\cK_\infty^\ac,A_{\F}^\dagger))={\rm Char}_{\cR[[\Gamma_\cK^\ac]]}\biggl(\frac{{\rm Sel}_{{\tt Gr},\emptyset}(\cK,\Tc)}{\cR[[\Gamma_\cK^\ac]]\cdot\mathcal{BF}^{\dagger,\ac}}\biggr)
\]
as ideals in $\cR[[\Gamma_\cK^\ac]]\otimes_{\bZ_p}\bQ_p$, and so the result follows from the combination of (\ref{eq:cor-ht}), (\ref{eq:Gr-rel}), and (\ref{eq:takechar}).
\end{proof}

The aforementioned $\cR[[\Gamma_\cK^{\rm ac}]]$-adic Gross--Zagier formula for $\mathfrak{Z}_\infty$ is the following.

\begin{cor}\label{cor:I-GZ}
Assume hypotheses {\rm (h0)--(h6)}. Then we have the equality
\[
(L_{p,1}^{\tt Hi}(\F^\dagger/\cK)_{\ac})=(\langle\mathfrak{Z}_\infty,\mathfrak{Z}_\infty\rangle_{\cK_\infty^{\ac},\cR}^{\rm cyc})
\]
as ideals of $\cR[[\Gamma_\cK^\ac]]\otimes_{\cR}\cR$. 
\end{cor}

\begin{proof}
Since ${\rm Sel}_{\tt Gr}(\cK,\Tc)$ has $\cR[[\Gamma_\cK^\ac]]$-rank one by Theorem~\ref{thm:HP-MC} and $\mathfrak{Z}_\infty$ is not $\cR[[\Gamma_\cK]]$-torsion, the regulator $\mathcal{R}_{\rm cyc}$ of $(\ref{eq:lambda-ht})$ satisfies
\[
(\langle\mathfrak{Z}_\infty,\mathfrak{Z}_\infty\rangle_{\cK_\infty^{\ac},\cR}^{\rm cyc})=\mathcal{R}_{\rm cyc}\cdot {\rm Char}_{\cR[[\Gamma_\cK^\ac]]}\biggl(\frac{{\rm Sel}_{\tt Gr}(\cK,\Tc)}{\cR[[\Gamma_\cK^\ac]]\cdot\mathfrak{Z}_\infty}\biggr)\cdot{\rm Char}_{\cR[[\Gamma_\cK^\ac]]}\biggl(\frac{{\rm Sel}_{\tt Gr}(\cK,\Tc)}{\cR[[\Gamma_\cK^\ac]]\cdot\mathfrak{Z}_\infty}\biggr)^\iota.
\]
By the ``functional equation'' of \cite[p.~1464]{howard-PhD-I}, the result thus follows from the combination of Theorem~\ref{thm:3.1.5} and the equality of characteristic ideals in Theorem~\ref{thm:HP-MC}.
\end{proof}



Now we can conclude the proof of our application to Greenberg's nonvanishing conjecture. 

\begin{thm}\label{thm:Gr-1}
Assume that:
\begin{itemize}
\item[(i)] $\cR$ is regular,
\item[(ii)] $\bar{\rho}_{\F}$ is irreducible, 
\item[(iii)]{} some specialization $\F_\phi$ is the $p$-stabilization of a newform $f\in S_2(\Gamma_0(N))$,
\item[(iv)] $N$ is squarefree,
\item[(v)] there are at least two primes $\ell\vert N$ at which $\bar{\rho}_{\F}$ is ramified. 
\end{itemize}	
If ${\rm Sel}_{\tt Gr}(\Q,T_{\F}^\dagger)$ has $\cR$-rank one
and the $\I$-adic height pairing
$\langle,\rangle_{\Q,\I}^{\rm cyc}$ is non-degenerate, then
\begin{equation}\label{eq:gen}
\frac{d}{ds}L_p^{\tt MTT}(\F_\phi,s)\biggr\vert_{s=k_\phi/2}\neq 0,
\nonumber
\end{equation}
for all but finitely many $\phi\in\mathcal{X}_a^o(\cR)$.
\end{thm}

\begin{proof}
Let $\phi\in\mathcal{X}_a^o(\cR)$ be such that $\F_\phi$ is the ordinary $p$-stabilization of a newform $f\in S_2(\Gamma_0(N))$. Let $\ell_1$ and $\ell_2$ be two distinct primes as in hypothesis (v), and choose an imaginary quadratic field $\cK$ such that the following hold:
\begin{itemize}
\item $\ell_1$ and $\ell_2$ are inert in $\cK$,
\item every prime dividing $N^+:=N/\ell_1\ell_2$ splits in $\cK$,
\item $p$ splits in $\cK$,
\item $L(f\otimes\epsilon_\cK,1)\neq 0$, where $\epsilon_\cK$ is the quadratic character corresponding to $\cK$. 
\end{itemize}

Note that the existence of $\cK$ is ensured by \cite{FH}, and that, so chosen, $\cK$ satisfies (\ref{ass:gen-H}) with $N^-=\ell_1\ell_2$. Now, the action of a complex conjugation $\tau$ combined with the restriction map induces an isomorphism
\begin{equation}\label{eq:dec}
{\rm Sel}_{\tt Gr}(\cK,T_{\F}^\dagger)\simeq{\rm Sel}_{\tt Gr}(\Q,T_{\F}^\dagger)\oplus{\rm Sel}_{\tt Gr}(\Q,T_{\F}^\dagger\otimes\epsilon_\cK),
\end{equation}
where the first and second summands are identified with the $+$ and $-$ eigenspaces for the action of $\tau$, respectively (see \cite[Lem.~3.1.5]{SU}). By Kato's work \cite{Kato295}, the nonvanishing of $L(f\otimes\epsilon_\cK,1)$ implies that ${\rm Sel}(\bQ,T_{f}\otimes\epsilon_\cK)$ is finite, and so by the control theorem for ${\rm Sel}_{\tt Gr}(\Q,T_{\F}^\dagger\otimes\epsilon_\cK)$ (see the exact sequence in  \cite[Cor.~3.4.3]{howard-invmath}) we conclude that ${\rm Sel}_{\tt Gr}(\Q,T_{\F}^\dagger\otimes\epsilon_\cK)$ is $\cR$-torsion, and so 
\[
{\rm rank}_\cR\;{\rm Sel}_{\tt Gr}(\cK,T_{\F}^\dagger)={\rm rank}_\cR\;{\rm Sel}_{\tt Gr}(\Q,T_{\F}^\dagger)=1	
\]
by (\ref{eq:dec}) and our assumption. In particular, since by (i)--(iv) we are assuming (h0)--(h3), and (h4)--(h6) hold by our choice of $\cK$,  Theorem~\ref{thm:converse-How} yields the non-triviality of the class $\mathfrak{Z}_0$, and so the element $\langle\mathfrak{Z}_0,\mathfrak{Z}_0\rangle_{\cK,\cR}^{\rm cyc}\in\cR$ is non-zero by our hypothesis of non-degeneracy.
	
Let $L_p^{\tt Hi}(\F^\dagger/\cK)_{\rm cyc}$ be the image of ${\rm tw}_{\Theta^{-1}}(L_p^{\tt Hi}(\F/\cK))$ under the natural projection $\cR[[\Gamma_\cK]]\twoheadrightarrow\cR[[\Gamma_\cK^{\rm cyc}]]$. By Theorem~\ref{thm:cyc-res}, for every $\phi\in\mathcal{X}_a^o(\cR)$ we have the factorization
\begin{equation}\label{eq:factor-cyc}
\phi(L_p^{\tt Hi}(\F^\dagger/\cK)_{\rm cyc})={\rm tw}_{\Theta_\phi^{-1}}(L_p^{\tt MTT}(\F_\phi))\cdot {\rm tw}_{\Theta_\phi^{-1}}(L_p^{\tt MTT}(\F_\phi\otimes\epsilon_\cK))
\end{equation}
up to a unit in $\phi(\cR)[[\Gamma^{\rm cyc}]]^\times$. Expand
\begin{align*}
\phi(L_p^{\tt Hi}(\F^\dagger/\cK)_{\rm cyc})&=L_{p,0}^{\tt Hi}(\F_\phi^\dagger/\cK)+L_{p,1}^{\tt Hi}(\F_\phi^\dagger/\cK)\cdot(\gamma_{\rm cyc}-1)+\cdots,\\
{\rm tw}_{\Theta_\phi^{-1}}(L_p^{\tt MTT}(\F_\phi))&=L_{p,0}^{\tt MTT}(\F_\phi^\dagger)+L_{p,1}^{\tt MTT}(\F_\phi^\dagger)\cdot(\gamma_{\rm cyc}-1)+\cdots,\\
{\rm tw}_{\Theta_\phi^{-1}}(L_p^{\tt MTT}(\F_\phi\otimes\epsilon_\cK))&=L_{p,0}^{\tt MTT}(\F_\phi^\dagger\otimes\epsilon_\cK)+L_{p,1}^{\tt MTT}(\F_\phi^\dagger\otimes\epsilon_\cK)\cdot(\gamma_{\rm cyc}-1)+\cdots,
\end{align*}
as power series in $\gamma_{\rm cyc}-1$, and note that by the $p$-adic Mellin transform we have
\[
\frac{d}{ds}L_p^{\tt MTT}(\F_\phi,s)\biggr\vert_{s=k_\phi/2}\neq 0\quad\Longleftrightarrow\quad L_{p,1}^{\tt MTT}(\F^\dagger_\phi)\neq 0
\]
(see \cite[(24)]{venerucci-p-conv}). The constant term $L_{p,0}^{\tt Hi}(\F_\phi^\dagger/\cK)\in\cR$ vanishes by Proposition~\ref{thm:hida-1}, and so the factorization $(\ref{eq:factor-cyc})$ yields the following equality up to unit in $\cO_\phi^\times$:
\begin{equation}\label{eq:desc}
L_{p,1}^{\tt Hi}(\F^\dagger_\phi/\cK)=L_{p,1}^{\tt MTT}(\F_\phi^\dagger)\cdot L_{p,0}^{\tt MTT}(\F^\dagger_\phi\otimes\epsilon_\cK). 
\end{equation}
Finally, since by definition $L_{p,1}^{\tt Hi}(\F^\dagger/\cK)\in\cR$ agrees with the image of the linear term $L_{p,1}^{\tt Hi}(\F^\dagger/\cK)_{\ac}$ in $(\ref{eq:2-exp})$ under the augmentation map $\cR[[\Gamma_\cK^{\ac}]]\rightarrow\cR$, from Corollary~\ref{cor:I-GZ} specialized at the trivial character of $\Gamma_\cK^\ac$ and $(\ref{eq:desc})$ we see that  
\begin{align*}
\langle\mathfrak{Z}_0,\mathfrak{Z}_0\rangle^{\rm cyc}_{\cK,\I}\neq 0\quad&\Longrightarrow\quad
L_{p,1}^{\tt Hi}(\F_\phi^\dagger/\cK)\neq 0,\quad\textrm{for almost all $\phi\in\mathcal{X}_a^o(\cR)$}\\\
&\Longrightarrow\quad L_{p,1}^{\tt MTT}(\F^\dagger_\phi)\neq 0,\quad\quad\textrm{for almost all $\phi\in\mathcal{X}_a^o(\cR)$},
\end{align*}
concluding the proof of Theorem~\ref{thm:Gr-1}.
\end{proof}


\bibliographystyle{amsalpha}
\bibliography{Iwasawa-big}

\end{document}